\documentclass{article}

\usepackage[english]{babel}
\usepackage[utf8]{inputenc}%coding of writteninput
\usepackage[T1]{fontenc}%vectorized fonts (cm-super package)

\usepackage{amsfonts}
\usepackage{amsmath, amsthm, amssymb}
\usepackage{MnSymbol}

\usepackage{enumitem}

\usepackage{graphicx}  
 
  \usepackage{pgfplots}
  \usepackage{xcolor}
  \usepackage{tikz}
  \usetikzlibrary{plotmarks}

\usepackage{hyperref}
\usepackage{cite}

\newtheorem{theorem}{Theorem}
\theoremstyle{plain}

\newtheorem{lemma}[theorem]{Lemma}
\theoremstyle{plain}

\newtheorem{remark}[theorem]{Remark}
\theoremstyle{plain}

\theoremstyle{plain}

\newcommand{\hp}{s}

\newcommand{\dx}{\,dx}
\newcommand{\ds}{\,ds}
\newcommand{\dt}{\,dt}

\author{
Harald Garcke\footnotemark[1],
Michael Hinze\footnotemark[2],
Christian Kahle\footnotemark[3]
}

\date{\today}

\title{Optimal Control of time-discrete two-phase flow driven by a
diffuse-interface model~\footnotemark[4]
}

\begin{document}

\maketitle

\renewcommand{\thefootnote}{\fnsymbol{footnote}}

\footnotetext[1]{Fakult\"at f\"ur Mathematik, Universit\"at Regensburg, 93040 Regensburg, Germany
(\texttt {Harald.Garcke@mathematik.uni-regensburg.de}).}

\footnotetext[2]{Schwerpunkt Optimierung und Approximation, 
Fachbereich Mathematik, 
Universit\"at Hamburg, 
Bundesstrasse 55, 
20146 Hamburg,
Germany (\texttt {Michael.Hinze@uni-hamburg.de}).} 

\footnotetext[3]{Lehrstuhl f\"ur Optimalsteuerung, 
Zentrum Mathematik,
Technische Universit\"at M\"unchen, 
Garching bei M\"unchen, Germany 
(\texttt {Christian.Kahle@ma.tum.de}).}

\footnotetext[4]
{
The authors gratefully acknowledge the financial support by the 
Deutsche Forschungsgemeinschaft (DFG) through the grants GA695/6-2 (first author) and HI689/7-1
(second and third author) within the priority program SPP1506 ``Transport processes at fluidic
interfaces''.
The third author additionally gratefully acknowledges the support by the DFG through the
International Research Training Group IGDK 1754 ``Optimization and Numerical Analysis for Partial Differential Equations with Nonsmooth Structures''.
}

\renewcommand{\thefootnote}{\arabic{footnote}}

\begin{abstract}

We propose a general control framework for two-phase flows with variable
densities in the diffuse interface formulation, where the distribution of the
fluid components is described by a phase field.
The flow is governed by the diffuse interface model proposed in
[Abels, Garcke, Gr\"un,  M3AS 22(3):1150013(40), 2012].
On the basis of the stable time discretization proposed in
[Garcke, Hinze, Kahle,
APPL NUMER MATH, 99:151--171, 2016]
we derive necessary
optimality conditions for the time-discrete and the fully discrete optimal
control problem.
We present numerical examples with distributed and boundary controls, and also
consider the case, where the initial value of the phase field serves as control
variable.

\end{abstract}

\section{Introduction}
In this paper we study a general discrete framework for control of two-phase
fluids governed by the thermodynamically consistent
diffuse interface model proposed in \cite{AbelsGarckeGruen_CHNSmodell}.
For the discretization we use the approach of
\cite{GarckeHinzeKahle_CHNS_AGG_linearStableTimeDisc}, where the
authors propose a time discretization scheme, that preserves this important
property in the time discrete setting and, using a post-processing step, also in
the fully discrete setting including adaptive mesh discretization.
As control actions we consider distributed control, Dirichlet boundary control,
and control with the initial condition of the phase field.

For the practical implementation we
adapt the adaptive treatment developed in \cite{GarckeHinzeKahle_CHNS_AGG_linearStableTimeDisc} to the optimal control
setting.
On the discrete level, special emphasis has to be be taken for the control with
the intial value of the phase field, since the distribution of its phases is an
outcome of the optimization procedure and thus a-priori unknown. In this case we
combine the variational discretization from \cite{HinzeVariDisk} with
error estimation techniques to find a good mesh
for the numerical representation of the a-priori unknown phase distribution.

Let us comment on related literature on time discretizations and control of
(two-phase) fluids.
For investigations of further time discretizations we refer to
\cite{Aland__time_integration_for_diffuse_interface,
Hintermueller_Keil_Wegner__OPT_CHNS_density,
Gruen_Klingbeil_CHNS_AGG_numeric,
GarckeHinzeKahle_CHNS_AGG_linearStableTimeDisc,
Tierra_Splitting_CHNS,
GuoLinLowengrub_numericalMethodForCHNS_Lowengrub,
GruenMetzger__CHNS_decoupled}.
Concerning optimal control and  feedback control of fluids there is a wide range
of literature available.
Here we only mention
\cite{Gunzburger_bndContr_NavierStokes,HinzeKunisch_BC_NS,
FursikovGunzburgerHou_BoundaryControlNS,
Berggren_flowControl_vorticityReduction,
Bewley_StabFlowChannel,Hinze_InstClosedLoopControlOfNS}.

Let us further comment on available literature for control of Cahn--Hilliard
multiphase flow systems.
In \cite{HintermuellerWegner_OptimContrCH}
distributed optimal control of the Cahn--Hilliard system with a non smooth 
double obstacle potential is proposed,
and in \cite{HintermuellerWegner_OptCHNS} this work is extended to time-discrete
two-phase flow given by a Cahn--Hilliard Navier--Stokes system with equal densities.  Both
works aim at existence of optimal controls and first order optimality
conditions.
In \cite{Hintermueller_Keil_Wegner__OPT_CHNS_density} the authors consider time
discrete optimal control of multiphase flows based on the diffuse interface model of \cite{AbelsGarckeGruen_CHNSmodell}.
This work aims at establishing existence of solutions and stationarity
conditions for control problems with free energies governed by the
double obstacle potential,
which is achieved through an appropriate limiting
process of control problems with smooth relaxed free energies.
The focus of the present work is different in that we consider numerical
analysis of the fully discrete problem, propose a tailored numerical adaptive
concept for the control problem, and present numerical examples which clearly
show the potential of our approach.

We also mention the work of
\cite{Banas_control_InterfaceEvolution_multiPhaseFlows},
where optimal control for a binary fluid, that is described by its density
distribution, is proposed.

Let us finally comment on feedback control approaches for multiphase flows.
Model predictive control is applied to the model from
\cite{AbelsGarckeGruen_CHNSmodell} in
\cite{HinzeKahle_instContr_IFAC_2013,Kahle_instContr_two_phase_OWR,kahle_dissertation}.

The paper is organized as follows. In Section \ref{sec:equs} we state the model
for the two-phase system and summarize assumptions that we require for the data.
In Section \ref{sec:TD} we state the time discretization scheme
proposed in
\cite{GarckeHinzeKahle_CHNS_AGG_linearStableTimeDisc} and summarize
properties of the scheme which we need in the present paper.
We formulate the time discrete optimization problem in Section
\ref{sec:TD}.
In Section \ref{sec:FullyDiscrete} we consider the optimal control problem in
the fully discrete setting and present numerical examples in
Section \ref{sec:num}.

\section{The governing equations}\label{sec:equs}

The two-phase flow is modeled by the diffuse interface model proposed in
\cite{AbelsGarckeGruen_CHNSmodell}.

\begin{align}
  \rho\partial_t v + \left( \left( \rho v + J\right)\cdot\nabla
  \right)v
  - \mbox{div}\left(2\eta Dv\right) + \nabla p &\nonumber\\
  -\mu\nabla \varphi - \rho K
  -f &= 0
  && \forall x\in\Omega,\, \forall t \in
  I,\label{eq:CHNSstrong1}\\
  -\mbox{div}(v) &= 0&&
  \forall x\in\Omega,\, \forall t \in I,\label{eq:CHNSstrong2}\\
  \partial_t \varphi + v \cdot\nabla \varphi - \mbox{div}(b\nabla \mu) &= 0&&
  \forall x\in\Omega,\, \forall t \in I,\label{eq:CHNSstrong3}\\
  -\sigma\epsilon \Delta \varphi + \frac{\sigma}{\epsilon}W'(\varphi) - \mu
  &= 0&& \forall x\in\Omega,\, \forall t \in I,\label{eq:CHNSstrong4}\\
  v(0,x) &= v_0(x)      &&    \forall x \in \Omega,\label{eq:CHNSstrongIC1}\\
  \varphi(0,x) &=  \varphi_0(x) &&\forall x \in
  \Omega,\label{eq:CHNSstrongIC2}\\
  v(t,x) &= g &&\forall x \in \partial \Omega,\, \forall t
  \in I,\label{eq:CHNSstrongBC1}\\
  \nabla \mu(t,x)\cdot \nu_\Omega =
  \nabla \varphi(t,x) \cdot \nu_\Omega &= 0  &&\forall x \in \partial
  \Omega,\, \forall t \in I.\label{eq:CHNSstrongBC2}
\end{align}
Here $\varphi$ denotes the phase field, $\mu$ the chemical potential, $v$ the
velocity field and $p$ the pressure. Furthermore
$J =-\frac{\rho_2-\rho_1}{2}b\nabla \mu$ is a diffuse flux for $\varphi$.

In addition $\Omega \subset \mathbb{R}^n,\, n\in \{2,3\}$, denotes an   open, convex and
polygonal ($n=2$) or polyhedral ($n=3$) bounded domain. Its outer unit normal is denoted as
$\nu_\Omega$, and $I=(0,T]$ with $0<T<\infty$ is a time interval.

The free energy density is denoted by $W$ and is assumed to be of double-well type with
exactly two minima at $\pm 1$.
For $W$ we use a splitting $W = W_+ + W_-$, where
$W_+$ is convex and $W_-$ is concave.

The density is denoted by $\rho = \rho(\varphi)$, fulfilling $\rho(-1) = \rho_1$
and $\rho(1) = \rho_2$, where $\rho_1,\rho_2$ denote the densities of the
involved fluids.
The viscosity is denoted by $\eta = \eta(\varphi)$, fulfilling $\eta(-1) =
\eta_1$ and $\eta(1) = \eta_2$, with individual fluid viscosities $\eta_1,\eta_2$.
The constant mobility is denoted by $b$. The gravitational force is
denoted by $K$.
By $Dv = \frac{1}{2}\left(\nabla v + (\nabla v)^t\right)$ we denote the symmetrized gradient.
The scaled surface tension is denoted by $\sigma$ and the interfacial
width is proportional to $\epsilon$.
We further have a volume force $f$ and
boundary data $g$, as well as an initial phase field $\varphi_0$ and a
solenoidal initial velocity field $v_0$.

Concerning results on existence of solutions for
\eqref{eq:CHNSstrong1}--\eqref{eq:CHNSstrongBC2} under different assumptions on
$W$ and $b$  we refer to
\cite{AbelsDepnerGarcke_CHNS_AGG_exSol,
AbelsDepnerGarcke_CHNS_AGG_exSol_degMob,
Gruen_convergence_stable_scheme_CHNS_AGG}.

\subsection*{Assumptions}
For the data of our
problem we assume:
\begin{enumerate}[label=(A\arabic{*})]
  \item \label{ass:psi_C2}
  $W:\mathbb{R}\to \mathbb{R}$ is twice continuously differentiable and is
  of double-well type, i.e. it has exactly two minima at $\pm1$ with values
  $W(\pm1) =0$.
  \item \label{ass:psi_boundedPoly}
  $W$ and its derivatives are polynomially bounded, i.e. there exists a $C>0$
  such that $|W(x)| \leq C(1+|x|^q)$, $|W'_+(x)| \leq C(1+|x|^{q-1})$,
  $|W'_-(x)| \leq  C(1+|x|^{q-1})$,
  $|W''_+(x)| \leq  C(1+|x|^{q-2})$,
  and $|W''_-(x)| \leq  C(1+|x|^{q-2})$ holds for some
  $q \in [2,4]$ if $n=3$ and $q \in [2,\infty)$ if $n=2$.
  \item \label{ass:rhoeta_linear}
  There exists $\varphi_a\leq -1$ and $\varphi_b\geq 1$, such that
  $\rho(\varphi) = \rho(\varphi_a)$ for $\varphi\leq
  \varphi_a$, and $\rho(\varphi) = \rho(\varphi_b)$ for $\varphi\geq \varphi_b$.
  For $\varphi_a<\varphi<\varphi_b$ the function $\rho(\varphi)$ is
  affine linear, i.e.
  $\rho(\varphi) = \frac{1}{2}\left((\rho_2-\rho_1)\varphi
  +(\rho_1+\rho_2)\right) $,
  and we define $\rho_\delta :=
  \frac{(\rho_2-\rho_1)}{2}$.
  
  Further,  $\eta(\varphi) = \eta(\varphi_a)$ for $\varphi\leq \varphi_a$, and
  $\eta(\varphi) = \eta(\varphi_b)$ for $\varphi\geq \varphi_b$.
  For $\varphi_a<\varphi<\varphi_b$ the function $\eta(\varphi)$ is affine linear,
  i.e.
  $\eta(\varphi) = \frac{1}{2}\left((\eta_2-\eta_1)\varphi
  +(\eta_1+\eta_2)\right) $.
   
  We define $\overline \rho > \underline \rho>0$, $\overline \eta \geq \underline
  \eta>0$ fulfilling
  \begin {itemize}
    \item[$\bullet$] $\overline \rho \geq \rho(\varphi)\geq\underline \rho>0$,
    \item[$\bullet$] $\overline \eta \geq \eta(\varphi)\geq\underline \eta>0$,
  \end{itemize}
  see Remark \ref{rm:phi_bounded}.
  \item \label{ass:phi_meanZero}
  The mean value of $\varphi$ is zero,
  i.e. there holds $\frac{1}{|\Omega|} \int_\Omega \varphi\,dx = 0$. This can be
  achieved by choosing the values indicating the pure phases accordingly and
  considering a shifted system if required.
  In this case the values $\pm1$ change to some other appropriate values.
\end{enumerate}

\begin{remark}\label{rm:freeEnergies}
The Assumptions \ref{ass:psi_C2}--\ref{ass:psi_boundedPoly} are for
example fulfilled by the polynomial free energy density
\begin{align*}
  W^{poly}(\varphi) = \frac{1}{4}\left(1-\varphi^2\right)^2.
\end{align*}
Another free energy density fulfilling these assumptions is the relaxed
double-obstacle free energy density given by
\begin{align}
  \lambda(y)& := \max(0,y-1) + \min(0,y+1),\nonumber\\
  \xi &:= \frac{1+2\hp+\sqrt{4\hp+1}}{2\hp},\nonumber\\
  \delta & := \frac{1}{2}\left(1-\xi^2\right) +
  \frac{\hp}{3}|\lambda(\xi)|^3,\nonumber\\
  W^{\hp}(y) & =
  \frac{1}{2}\left(
  1-(\xi y)^2\right)
  +\frac{\hp}{3}|\lambda(\xi y)|^3
  + \delta
  \label{eq:freeEnergy}
\end{align}
where $\hp\gg 0$ denotes a relaxation parameter.
$W^{\hp}$ can be understood as a relaxation of the
double-obstacle free energy density
\begin{align*}
  W^{\infty}(\varphi) =
  \begin{cases}
\frac{1}{2}\left(1-\varphi^2\right) & \mbox{ if } |\varphi|\leq 1,\\
0           & \mbox{ else},
\end{cases}
\end{align*}
which is proposed in \cite{OonoPuri_DoubleObstacelPotential,BloweyElliott_I} to
model phase separation. We note that here we use a cubic penalisiation to obtain the required regularity from
\ref{ass:psi_C2} and that $\xi$ is chosen such that $W^{\hp}$ takes its
minima at $\pm 1$ and $\delta$ is such that $W^{\hp}(\pm1) = 0$.

In the numerical examples of this work we use the free energy density $W \equiv
W^{\hp}$.
For this choice the splitting into convex and concave part reads
\begin{align*}
  W_+(\varphi) &= \hp\frac{1}{3}|\lambda(\xi\varphi)|^3,
  &
  W_-(\varphi) &= \frac{1}{2}(1-(\xi\varphi)^2) + \delta.
\end{align*}
\end{remark}

\begin{remark}\label{rm:phi_bounded}
For the weak formulation of \eqref{eq:CHNSstrong1}--\eqref{eq:CHNSstrongBC2} we
later require affine linearity of $\rho$ on the image of $\varphi$. The affine linearity of
$\eta$ is assumed for simplicity. Note that in view of Assumption
\ref{ass:rhoeta_linear}, this essentially implies a bound on $\varphi$, namely $\varphi \in
(\varphi_a,\varphi_b)$ as stated in Assumption~\ref{ass:rhoeta_linear}.

Using $W^{\hp}$ as free energy density we argue, that for
$\hp$ sufficiently large 
(see \cite[Rem. 6]{GarckeHinzeKahle_CHNS_AGG_linearStableTimeDisc})  
 $|\varphi| \leq 1+\theta$ holds, with $\theta$ sufficiently small, and  in
\cite{Kahle_Linfty_bound} it is shown for the Cahn--Hilliard equation without transport, that for the
energy \eqref{eq:freeEnergy} in fact 
$\|\varphi\|_{L^\infty(\Omega)} \leq 1+C\hp^{-1/2}$ holds.

In a general setting one might use a nonlinear dependence between
$\varphi$ and $\rho$, see e.g. \cite{AbelsBreit_weakSolution_nonNewtonian_DifferentDensities},
or choose a cut-off procedure as
proposed in \cite{Gruen_convergence_stable_scheme_CHNS_AGG,Tierra_Splitting_CHNS}.

Anyway, since we later require linearity of $\rho$ on the image of $\varphi$ we
state Assumption \ref{ass:rhoeta_linear} and note that this assumption is fulfilled in our numerical
examples in Section~\ref{sec:num}.
\end{remark}

%____________________________________________________________
%____________________________________________________________
%___________________Notation_________________________________
%____________________________________________________________
%____________________________________________________________

\subsection*{Notation}
%label identifier: NOT

We use the conventional notation for Sobolev and Hilbert Spaces, see e.g.
\cite{Adams_SobolevSpaces}.
With $L^p(\Omega)$, $1\leq p\leq \infty$, we denote the space of measurable functions on $\Omega$,
whose modulus to the power $p$ is Lebesgue-integrable. $L^\infty(\Omega)$ denotes the space of
measurable functions on $\Omega$, which are essentially bounded.
For $p=2$ we denote by
$L^2(\Omega)$ the space of square integrable functions on $\Omega$
with inner product $(\cdot,\cdot)$ and norm $\|\cdot \|$.
By $W^{k,p}(\Omega)$, $k\geq 1, 1\leq p\leq \infty$, we denote the Sobolev space of functions
admitting weak derivatives up to order $k$ in $L^p(\Omega)$.
If $p=2$ we write $H^k(\Omega)$.

For $f \in H^1(\Omega)^n$ we introduce the continuous trace operator
$\gamma:H^1(\Omega)^n \to H^{\frac{1}{2}}(\partial\Omega)^n$
as $\gamma f := f|_{\partial\Omega}$.
We further note that for $g \in H^{\frac{1}{2}}(\partial\Omega)^n$ with $g\cdot \nu_\Omega = 0$
there exists $\widetilde g \in H^1(\Omega)^n, (div \widetilde g,q) = 0 \forall q \in L^2(\Omega)$ with
$\gamma \widetilde g = g$ and $\|\widetilde g\|_{H^1(\Omega)^n} \leq C
\|g\|_{H^{\frac{1}{2}}(\partial \Omega)^n}$, where $C$ is independent of $g$. 

The subspace $H^1_0(\Omega)^n \subset H^1(\Omega)^n$
denotes the set of functions with vanishing boundary trace.
We further set
\begin{align*}
  L^2_{(0)}(\Omega) = \{v\in L^2(\Omega)\,|\, (v,1) = 0\},
\end{align*}
and with
\begin{align*}
  H_{\sigma}(\Omega) = \{v\in H^1(\Omega)^n\,|\, (\mbox{div}(v),q) =
  0\,\forall q\in L^2(\Omega) \}
\end{align*}
we denote the space of all weakly solenoidal $H^1(\Omega)$ vector fields.
We
stress that there is no correspondence between the subscript $\sigma$ and the
scaled surface tension. We denote both terms using $\sigma$ since these are
standard notations.
We further introduce
\begin{align*}
  H_{0,\sigma}(\Omega) = H^1_0(\Omega)^n \cap H_\sigma(\Omega).
\end{align*}

For $u\in L^q(\Omega)^n$, $q>2$ if $n=2$, $q\geq3$ if $n=3$, and $v,w \in
H^1(\Omega)^n$ we introduce the trilinear form
\begin{align}
  a(u,v,w) =
  \frac{1}{2}\int_\Omega\left( \left(u\cdot\nabla\right)v \right)w\,dx
  -\frac{1}{2}\int_\Omega\left( \left(u\cdot\nabla\right)w \right)v\,dx.
  \label{eq:NOT:trilinearForm}
\end{align}
Note that there holds $a(u,v,w) = -a(u,w,v)$, and especially $a(u,v,v) = 0$.
We have the following stability estimate by H\"older inequalities and Sobolev
embedding
\begin{align*}
  |a(u,v,w)| \leq C\|u\|_{L^q(\Omega)}\|v\|_{H^1(\Omega)}\|w\|_{H^1(\Omega)}.
\end{align*}

%NOTE: reformulation to antisymmetry of NS is not affected by boundary data

For a square summable series of functions $\left(f_m\right)_{m=1}^{M} \in V^M$,
where $(V,\|\cdot\|_V)$ is a normed vector space, we introduce the notation
$\|(f^m)_{m=1}^M\|_V^2 = \sum_{m=1}^M\|f_m\|_V^2 $.

%____________________________________________________________
%____________________________________________________________
%___________________Time_discrete_setting____________________
%____________________________________________________________
%____________________________________________________________

\section{The time-discrete setting}\label{sec:TD}
In \cite{GarckeHinzeKahle_CHNS_AGG_linearStableTimeDisc} existence of time
discrete weak solutions for \eqref{eq:CHNSstrong1}--\eqref{eq:CHNSstrong4} is
shown for the case of $g=0$ and $f=0$.
In this section we formulate a time discrete optimization problem for
\eqref{eq:CHNSstrong1}--\eqref{eq:CHNSstrong4}, where we use $g, f$, and
$\varphi^0$ as controls, and show existence of solutions together with first
order optimality conditions.

\bigskip

Let
$0=t_0<t_1<\ldots<t_{m-1}<t_m<t_{m+1}<\ldots<t_M=T$
denote an equidistant subdivision of the
interval $\overline I = [0,T]$ with $\tau_{m+1}-\tau_m \equiv \tau$ and
sub intervals $I_0 = \{0\}$, $I_m = (t_{m-1},t_m],\,m=1,\ldots,M$.
From here onwards the superscript $m$ denotes the corresponding variables at
time instance $t_m$, e.g. $\varphi^m := \varphi(t_m)$.
For functions $f \in L^2(0,T,V)$ we introduce
$f^m := \strokedint_{I_m} f(t)\dt \in V$. Note that this can be seen as a
discontinuous Galerkin approximation using piecewise constant values.

We now introduce the optimal control problem under consideration.
For this purpose we interpret  $\varphi_0, f,$ and $g$
as sought control that we intend to choose, such that the
corresponding phase field $\varphi^M$ is close to a desired phase field
$\varphi_d$ in the mean square sense.
If $\varphi_d$ is the measurement of a real world system, then finding
$\varphi^0$ such that the corresponding phase field $\varphi^M$ is
close to $\varphi_d$ resembles an inverse problem.

We denote by $u\in U$  the control, where
\begin{align*}
  U =  U_I \times U_V \times U_B = \mathcal K
  \times
  L^2(0,T;\mathbb{R}^{u_v})
  \times
  L^2(0,T;\mathbb{R}^{u_b})
\end{align*}
is the space of
controls, where 
\begin{align*}
  \mathcal K := \{ v\in H^1(\Omega) \, | \, \int_\Omega v\dx = 0,\, |v| \leq 1  \} \subset
  H^1(\Omega) \cap L^\infty(\Omega)
\end{align*}
denotes the space of admissible initial phase fields. 

By
\begin{align*}
  \mathcal{B}:U\to H^1(\Omega)\cap L^\infty(\Omega)
  \times
  L^2(0,T;L^2(\Omega)^n)
  \times
  L^2(0,T;(H^{1/2}(\partial\Omega))^n)
\end{align*}
we denote the  linear and bounded control operator, which consists of three
components, i.e. $\mathcal{B} = [B_I,B_V,B_B]$,
where $B_I (u_I,u_V,u_B) \equiv B_I u_I := u_I$,
which is the initial phase field for the system,
$B_V (u_I,u_V,u_B) \equiv B_V u_V$ with
$B_Vu_V(t,x) = \sum_{l=1}^{u_v} f_l(x)u_V^l(t)$ where
$f_l \in L^2(\Omega)^n$ are given functions, which is a volume force acting on
the fluid inside $\Omega$,
and $B_B (u_I,u_V,u_B) \equiv B_B u_B$, with
$B_Bu_B(t,x) = \sum_{l=1}^{u_b} g_l(x)u_B^l(t)$ where  $g_l \in
H^{1/2}(\partial\Omega)^n$ denote given functions, and this is a boundary force
acting on the fluid as Dirichlet boundary data.
To obtain a solenoidal velocity field, $B_Bu_B$ has to fulfill the
compatibility condition $\int_{\partial\Omega} B_Bu_B \cdot\nu_\Omega\ds = 0$,
and in the following for simplicity we assume $g_l\cdot \nu_\Omega = 0$, $l=1,\ldots,u_b$,
point wise.

Given a triple
$(\alpha_I,\alpha_V,\alpha_B)$ of non negative values with
$\alpha_I + \alpha_V+\alpha_B = 1$ we introduce an inner product  for
$u = (u_I,u_V,u_B)\in U$ and $v = (v_I,v_V,v_B)\in U$ by
\begin{align}
  \label{eq:TD:innerU}
  (u,v)_U = \alpha_I(\nabla u_I,\nabla v_I)_{L^2(\Omega)}
  +\alpha_V (u_V,v_V)_{L^2(0,T;\mathbb{R}^{u_v})}
  +\alpha_B (u_B,v_B)_{L^2(0,T;\mathbb{R}^{u_b})}
\end{align}
and the norm $\|u\|_U^2 = (u,u)_U $.

We use the convention, that  $\alpha_\star = 0$, $\star \in \{I,V,B\}$,
means, that we do not apply this kind of control. If $\alpha_I = 0$ we use
$\varphi^0$ as given data, if $\alpha_B= 0$, we assume no-slip boundary data for
$v$. For notational convenience, in the following we assume $\alpha_\star \not = 0 $ for all
$\star \in \{I,V,B\}$.

We stress, that we do not discretize the
control in time, although the state equation is time discrete. Thus we follow
the concept of variational discretization \cite{HinzeVariDisk}. Anyway, the control
is discretized implicitly in time by the adjoint equation that we will
derive later.
We also note, that in view of the state equation, this allows us to dynamically
adapt the time step size $\tau$ to the flow condition without changing the
control space.

Following
\cite{GarckeHinzeKahle_CHNS_AGG_linearStableTimeDisc} we propose the following time discrete counterpart of
\eqref{eq:CHNSstrong1}--\eqref{eq:CHNSstrongBC2}:\\
Let $u\in U$
and $v_0 \in  H_\sigma(\Omega) \cap L^\infty(\Omega)$ be given.

\medskip

\noindent \textit{Initialization for $m=1$:}\\
Set $\varphi^0 =u_I$ and $v^0=v_0$.\\
Find $\varphi^1 \in H^1(\Omega)\cap L^\infty(\Omega)$, $\mu^1\in
W^{1,3}(\Omega)$,
$v^1 \in H_\sigma(\Omega)$, 
with $\gamma(v^1) = B_Bu_B^1$,
such that for all $w\in H_{0,\sigma}(\Omega)$, $\Phi \in H^1(\Omega)$,
and $\Psi \in H^1(\Omega)$ it holds
\begin{align}
  \frac{1}{\tau}\int_\Omega \left(\frac{1}{2}(\rho^1 + \rho^0)v^1-\rho^0v^0\right) w \dx
  +a(\rho^1v^0+J^1, v^1, w)&\nonumber\\
  +\int_{\Omega} 2\eta^1  Dv^{1}:Dw\dx
  -\int_\Omega \mu^{1}\nabla \varphi^0 w + \rho^0 K w\dx
  -\left<B_Vu_V^1,w\right>_{H^{-1}(\Omega),H_0^1(\Omega)} &= 0,
  \label{eq:TD:chns1_solenoidal_init}\\
  \frac{1}{\tau}\int_\Omega (\varphi^{1}-\varphi^0) \Psi \dx +
  \int_\Omega(v^{0}\cdot \nabla \varphi^{0}) \Psi \dx
  +  \int_\Omega b\nabla \mu^{1}\cdot\nabla \Psi\dx &=0,
  \label{eq:TD:chns2_solenoidal_init}\\
  \sigma \epsilon\int_\Omega\nabla \varphi^{1}\cdot\nabla \Phi\dx
  - \int_\Omega \mu^{1}\Phi\dx + \frac{\sigma}{\epsilon}\int_\Omega
  (W_+^\prime(\varphi^{1})+W_-^\prime(\varphi^0))\Phi\dx &= 0,
  \label{eq:TD:chns3_solenoidal_init}
\end{align}
where $ J^1 := -\rho_\delta b\nabla \mu^1$.

\medskip

\noindent \textit{Two-step scheme for $m> 1$:} \\
Given $\varphi^{m-2}\in H^1(\Omega)\cap L^\infty(\Omega)$,
$\varphi^{m-1}\in H^1(\Omega)\cap L^\infty(\Omega)$,
$\mu^{m-1} \in W^{1,3}(\Omega)$,
$v^{m-1} \in H_\sigma(\Omega)$,\\
find
$v^{m} \in H_\sigma(\Omega)$, $\gamma(v^m) = B_Bu_B^m$,
$\varphi^{m}\in H^1(\Omega)\cap L^\infty(\Omega)$,
$\mu^{m}\in W^{1,3}(\Omega)$
such that for all $w\in H_{0,\sigma}(\Omega)$,
$\Psi \in H^1(\Omega)$,
and $\Phi \in H^1(\Omega)$ it
holds
\begin{align}
  \frac{1}{\tau}\int_\Omega\left(\frac{\rho^{m-1}+\rho^{m-2}}{2} v^m -
  \rho^{m-2}v^{m-1}\right) w\dx+\int_{\Omega} 2\eta^{m-1}Dv^{m}:Dw\dx
  &\nonumber\\
  +a(\rho^{m-1}v^{m-1}+J^{m-1},v^{m},w)&\nonumber \\
  -   \int_\Omega \mu^{m}\nabla \varphi^{m-1} w + \rho^{m-1} K w\dx
  -\left<B_Vu_V^m,w\right>_{H^{-1}(\Omega),H_0^1(\Omega)} &=
  0,\label{eq:TD:chns1_solenoidal}\\
  \int_\Omega \frac{\varphi^{m}-\varphi^{m-1}}{\tau} \Psi \dx +
  \int_\Omega(v^{m}\cdot \nabla \varphi^{m-1}) \Psi \dx +  \int_\Omega
  b\nabla \mu^{m}\cdot\nabla \Psi\,dx &=0
  ,\label{eq:TD:chns2_solenoidal}\\
  \sigma \epsilon\int_\Omega\nabla \varphi^{m}\cdot\nabla \Phi \dx
  - \int_\Omega \mu^{m}\Phi \dx
  +\frac{\sigma}{\epsilon}\int_\Omega
  (W_+^\prime(\varphi^{m})+W_-^\prime(\varphi^{m-1}))\Phi \dx &= 0,\label{eq:TD:chns3_solenoidal}
\end{align}
where $ J^{m-1} := -\rho_\delta b \nabla \mu^{m-1}$.
We further use the abbreviations $\rho^{m} := \rho(\varphi^{m})$ and $\eta^{m}
:= \eta(\varphi^{m})$.

We note that in
\eqref{eq:TD:chns1_solenoidal}--\eqref{eq:TD:chns3_solenoidal} the only
nonlinearity arises from $W_+'$ and thus only the equation
\eqref{eq:TD:chns3_solenoidal} is nonlinear.
A similar argumentation holds for
\eqref{eq:TD:chns1_solenoidal_init}--\eqref{eq:TD:chns3_solenoidal_init}.
The regularity $\nabla \mu^{m-1} \in L^3( \Omega)$ is required for the trilinear
form $a(\cdot,\cdot,\cdot)$, see \eqref{eq:NOT:trilinearForm}.

%%%%%%%%%%%%%%%%%%%%%%%%%%%%%%%%%%%%%%%%%%%%%%%%%%%%%%%%%%%%%%%%

\begin{remark}
We note that \eqref{eq:TD:chns1_solenoidal}--\eqref{eq:TD:chns3_solenoidal}
is a two-step scheme for the phase field variable $\varphi$,
and thus we need an initialization as proposed in
\eqref{eq:TD:chns1_solenoidal_init}--\eqref{eq:TD:chns3_solenoidal_init}.
Here, as in \cite{GarckeHinzeKahle_CHNS_AGG_linearStableTimeDisc} the
sequential coupling of
\eqref{eq:TD:chns2_solenoidal_init}--\eqref{eq:TD:chns3_solenoidal_init} and
\eqref{eq:TD:chns1_solenoidal_init}  is used as proposed in \cite{KayStylesWelford}.

Another variant might be to require initial data on time instance $t_{-1}$ for
the phase field and at $t_0$ for the velocity field. Equations
\eqref{eq:TD:chns2_solenoidal}--\eqref{eq:TD:chns3_solenoidal} can than be
solved for $\varphi^0$ and $\mu^0$ to obtain initial values, see
\cite{Hintermueller_Keil_Wegner__OPT_CHNS_density}.

Since we are later also interested in control of the initial value $\varphi_0$
we propose the initialization scheme
\eqref{eq:TD:chns1_solenoidal_init}--\eqref{eq:TD:chns3_solenoidal_init} here.
\end{remark}

\begin{theorem}\label{thm:TD:exSol_initStep}
Let $v^0\in H_\sigma(\Omega) \cap L^\infty(\Omega)^n$
and $u\in U$ be given data.

Then there exists a unique solution $(v^1,\varphi^1,\mu^1)$ to
\eqref{eq:TD:chns1_solenoidal_init}--\eqref{eq:TD:chns3_solenoidal_init}, and it holds
\begin{align}
  \|v^1\|_{H^{1}(\Omega)^n} +
  &\|\varphi^1\|_{H^2(\Omega)} +
  \|\mu^1\|_{H^{2}(\Omega)}\nonumber \\
  & \leq
  C_1(v_0)
  C_2\left(\|u_I\|_{H^1(\Omega)},
  \|B_Vu_V^1\|_{L^2(\Omega)^n},
  \|B_Bu_B^1\|_{H^{\frac{1}{2}}(\partial\Omega)^n}\right)
  \label{eq:TD:exSol_initStep:stab}
\end{align}
and
$\varphi^1,\mu^1$ can be found be Newton's method.
The constant $C_2$ depends polynomially on its arguments.
\end{theorem}

\begin{proof}
The existence  of $(\varphi^1,\mu^1)\in H^1(\Omega) \times
H^1(\Omega)$ follows from  (\cite{HintermuellerHinzeTber}). There the corresponding system without
the transport term $v^0\nabla u_I$ is
analyzed. This term is a given volume force, that can be incorporated in a
straightforward manner. From this we directly obtain the stability inequality
\begin{align*}
  \|\varphi^1\|_{H^1(\Omega)} + \|\mu^1\|_{H^1(\Omega)}
  \leq
  C_1(v^0)C_2(\|u_I\|_{H^1(\Omega)}).
\end{align*}

Since $|W^\prime_+(\varphi)|\leq C(1+|\varphi|^q|)$, $q\leq 3$ we have
$W^\prime_+(\varphi) \in L^2(\Omega)$ and by $L^2$ regularity theory we have
$\varphi^1 \in H^2(\Omega)$ and
\begin{align*} 
  \|\varphi^1\|_{H^2(\Omega)}
  \leq C(\|\mu^1\|_{H^1(\Omega)}, \|\varphi^1\|_{H^1(\Omega)}, \|u_I\|_{H^1(\Omega)}).
\end{align*}

We further have $v^0\nabla u_I \in L^{2}(\Omega)$ and thus we have
$\mu^1\in H^{2}(\Omega)$ and the stability inequality
\begin{align*}
  \|\mu^1\|_{H^2(\Omega)} \leq
  C_1(v_0)C_2(\|u_I\|_{H^1(\Omega)},\|\varphi^1\|_{H^1(\Omega)}).
\end{align*}

Convergence of Newton's method directly follows from
\cite{HintermuellerHinzeTber}. Note that the only nonlinearity $W^\prime_+$ is
monotone.

With $v^0,\varphi^1,u_I$, and $\mu^1$ given data,
\eqref{eq:TD:chns1_solenoidal_init}
defines a coercive and continuous bilinear form on $H_\sigma$ and thus existence and stability of a
solution follows from Lax-Milgram's theorem.
This uses the antisymmetry of the trilinear form $a$ and Korn's inequality.
\end{proof}

\begin{theorem}\label{thm:TD:exSol_twoStep}
Let $v^{m-1}\in  H_{\sigma}(\Omega)$,
$\varphi^{m-2}\in H^1(\Omega)\cap L^\infty(\Omega)$, $\varphi^{m-1}\in H^1(\Omega)\cap
L^\infty(\Omega)$, and $\mu^{m-1}\in W^{1,3}(\Omega)$, be given data.
Then there exists a unique solution  $(v^m,\varphi^m,\mu^m)$ to
\eqref{eq:TD:chns1_solenoidal}--\eqref{eq:TD:chns3_solenoidal}.

It further holds $\varphi^m\in H^2(\Omega)$ and if additionally $\varphi^{m-1} \in W^{1,3}(\Omega)$
we have $\mu^m \in H^2(\Omega)$ and  
 the stability inequality
\begin{align*}
  \|v^m\|_{H^{1}(\Omega)^n}
  + &\|\mu^m\|_{H^2(\Omega)}
  + \|\varphi^m\|_{H^2(\Omega)}\\
  \leq&
  C\left(
  \vphantom{\|v_{\partial\Omega}\|_{H^{\frac{1}{2}}(\partial\Omega)}}
  \|v^{m-1}\|_{H^{1}(\Omega)^n},
  \|\varphi^{m-1}\|_{W^{1,3}(\Omega)},
  \|\varphi^{m-2}\|_{H^1(\Omega)},\right.\\
  & \quad\left.
  \|B_Vu_V^m\|_{L^{2}(\Omega)^n},
  \|B_Bu_B^m\|_{H^{\frac{1}{2}}(\partial\Omega)^n}
  \right),
\end{align*}
holds.
The constant $C$ depends polynomially on its arguments.

\end{theorem}
\begin{proof}
In \cite{GarckeHinzeKahle_CHNS_AGG_linearStableTimeDisc} the existence for
$\gamma(B_Bu_B^m) = 0$ and $B_Vu_V^m = 0$ is shown using a Galerkin approach.
The additional volume force is incorporated in a straight forward manner, and
the boundary data $B_Bu_B^m$ can be introduced by investigating a shifted
system, see Theorem \ref{thm:TD:exSol_initStep}.

We define $e := \widetilde{B_Bu_B^m}$ and
use $w = v^m - e$ as test function in \eqref{eq:TD:chns1_solenoidal},
$\Psi = \mu^m$ as test function in \eqref{eq:TD:chns2_solenoidal}, and
$\Phi = \tau^{-1}(\varphi^m-\varphi^{m-1})$ as test function in
\eqref{eq:TD:chns3_solenoidal}, and add the resulting equations.
Using the
properties of $W^\prime_+$ and $W^\prime_-$ we obtain (compare
\cite[Thm. 3]{GarckeHinzeKahle_CHNS_AGG_linearStableTimeDisc})
\begin{align*}
  &E(v^m,\varphi^m,\varphi^{m-1})
  + \frac{1}{2}\int_\Omega \rho^{m-2}|v^{m}-v^{m-1}|^2\dx
  + 2\tau \int_\Omega\eta^{m-1}|Dv^m|^2\dx\\
  &+\tau \int_\Omega b|\nabla \mu^m|^2\dx
  +\frac{\sigma\epsilon}{2}\|\nabla \varphi^m-\nabla \varphi^{m-1}\|^2\\
  &\leq
  E(v^{m-1},\varphi^{m-1},\varphi^{m-2})
  +\int_\Omega\left(\frac{\rho^{m-1}+\rho^{m-2}}{2} v^m-\rho^{m-2}v^{m-1}\right)
  e\dx\\
  &+ \tau a(\rho^{m-1}v^{m-1}+J^{m-1},v^m,e)
  + 2\tau\int_\Omega \eta^{m-1}Dv^m:De\dx\\
  &-\tau\int_\Omega \mu^m\nabla \varphi^{m-1}e\dx
  +(\rho^{m-1}K,v^m-e)
  +\tau(B_Vu_V^m,v^m-e)_{L^{2}(\Omega)^n}.
\end{align*}
By using the inequalities of H\"older, Korn and Young, together with Assumption
\ref{ass:rhoeta_linear} and the stability of the extension operator
$\widetilde\cdot$ the claim follows.
The regularity $\varphi^m,\mu^m\in H^2(\Omega)$ follow as in the proof of
Theorem~\ref{thm:TD:exSol_initStep}, but now using $\nabla \varphi^{m-1} \in L^3(\Omega)$ and
$v^m\in H_\sigma \hookrightarrow L^6(\Omega)$.
\end{proof}

Let us next introduce the optimization problem under investigation. For this we
first rewrite
\eqref{eq:TD:chns1_solenoidal_init}--\eqref{eq:TD:chns3_solenoidal} in a compact
and abstract  form and introduce
\begin{align*}
  Y :=&
  H_{\sigma}(\Omega)^M
  \times
  \left(H^1(\Omega)\cap L^\infty(\Omega)\right)^M
  \times
  W^{1,3}(\Omega)^M,\\
  Y_0 :=&
  H_{0,\sigma}(\Omega)^M
  \times
  \left(H^1(\Omega)\cap L^\infty(\Omega)\right)^M
  \times
  W^{1,3}(\Omega)^M,\\
  y :=&
  (v^m,\varphi^m,\mu^m)_{m=1}^M \in Y,\\
  Z :=&
  \left(
  H_{0,\sigma}(\Omega)^M
  \times
  H^1(\Omega)^M \times
  H^1(\Omega)^M \right)^\star
\end{align*}
\begin{align}
  e : Y_0\times U \to  Z, \nonumber\\
  e(y_0,u)= 0\label{eq:TD:eyu_eq_0}
\end{align}
The operator $e$ is defined as follows
\begin{align*}
  &\left<\tilde y,e(y_0,u)\right>_{Z^\star,Z}:=\\
  &
  \tau^{-1}\left(
  \frac{1}{2}(\rho^1+\rho^0)(v_0^1+\widetilde{B_Bu_B^1}) -
  \rho^0v^0,\tilde v^1\right)
  +a(\rho^1v^0+J^1,v_0^1+\widetilde{B_Bu_B^1},\tilde v^1) \\
  &  + (2\eta^1D(v_0^1+\widetilde{B_Bu_B^1}),D\tilde v^1)
  - (\mu^1\nabla u_1+ \rho^0 K,\tilde v^1) \\
  &-(B_Vu_V^1,\tilde v^1)\\
  &+ \tau^{-1}(\varphi^1-u_1,\tilde \varphi^1)
  + (v^0\nabla u_1,\tilde \varphi^1)
  +(b\nabla \mu^1,\nabla \tilde \varphi^1)\\
  &+\sigma\epsilon(\nabla \varphi^1,\nabla \tilde \mu^1)
  -(\mu^1,\tilde \mu^1)\\
  &+\sigma\epsilon^{-1}
  (W^\prime_+(\varphi^1)+W^\prime_-(u_1),\tilde \mu^1)\\
  +\sum_{m=2}^M &\left[
  \tau^{-1}\left(\frac{1}{2}(\rho^{m-1}+\rho^{m-2})(v_0^m+\widetilde{B_Bu_B^m})
  - \rho^{m-2}(v_0^{m-1}+\widetilde{B_Bu_B^{m-1}}),\tilde v^m\right)\right.\\
  &+a(\rho^{m-1}(v_0^{m-1}+\widetilde{B_Bu_B^{m-1}})+J^{m-1},v_0^m+\widetilde{B_Bu_B^m},\tilde v^m) \\
  &  + (2\eta^{m-1}D(v_0^m+\widetilde{B_Bu_B^m}),D\tilde v^m)
  - (\mu^m\nabla \varphi^{m-1} + \rho^{m-1} K,\tilde v^m) \\
  &-(B_Vu_V^m,\tilde v^m)\\
  &+ \tau^{-1}(\varphi^m-\varphi^{m-1},\tilde \varphi^m)
  + ((v_0^m+\widetilde{B_Bu_B^m})\nabla \varphi^{m-1},\tilde \varphi^m)
  +(b\nabla \mu^m,\nabla \tilde \varphi^m)\\
  &+\sigma\epsilon(\nabla \varphi^m,\nabla \tilde \mu^m)
  -(\mu^m,\tilde \mu^m)\\
  &\left.+\sigma\epsilon^{-1}
  (W^\prime_+(\varphi^m)+W^\prime_-(\varphi^{m-1}),\tilde \mu^m)
  \vphantom{\frac{1}{2}}\right]\\
\end{align*}
with $y_0 := (v_0^m,\varphi^m,\mu^m)_{m=1}^M \in Y_0$,
and $\tilde y = ((\tilde v^m)_{m=1}^M,  (\tilde\varphi^m)_{m=1}^M,
(\tilde\mu^m)_{m=1}^M)))) \in Z^\star$.
Here again $\rho^m := \rho(\varphi^m)$, $\eta^m := \eta(\varphi^m)$ and
especially $\rho^0 := \rho(u_I)$, $\eta^0 := \eta(u_I)$.

%%%%%%%%%%%%%%%%%%%%%%%%%%%%%%%%%%%%%%%%%%%%%%%%%%%%%%%%%%%%%%%%

Now the time-discrete optimization problem under
investigation is given as
\begin{equation}  
  \begin{aligned}
    \min_{u\in U} J( (\varphi^m)_{m=1}^M,u) =
    \frac{1}{2}&\|\varphi^M-\varphi_d\|^2_{L^2(\Omega)}\\
    & +\frac{\alpha}{2}
    \left(
    \alpha_I
    \int_\Omega \frac{\epsilon}{2} |\nabla u_I|^2 +
    \epsilon^{-1}W_u(u_I)\dx
    \right.\\
    &\left.\vphantom{\int_\Omega}
    + \alpha_V \|u_V\|^2_{L^2(0,T;\mathbb{R}^{u_v})}
    + \alpha_B \|u_B\|^2_{L^2(0,T;\mathbb{R}^{u_b})}
    \right)\\
    s.t. \,\, e(y,u) = 0.
  \end{aligned}
  \tag{$\mathcal P$}\label{prob:TD:Ptau}
\end{equation}
Here $\varphi_d \in L^2(\Omega)$
is a given desired phase field, and $\alpha>0$ is a weight for the control
cost.
For the control cost of the initial value we use the well-known Ginzburg--Landau
energy of the phase field $u_I$ with interfacial thickness $\epsilon$. 
Here we use the double obstacle free energy density 
$W_u \equiv W^\infty$ given in Remark~\ref{rm:freeEnergies}.
In our numerical examples it is
advantageous to use this non-smooth free energy density instead of the smoother one used for the
simulation.

\begin{theorem}\label{thm:TD:exSol_pde}
Let $v^0 \in H_{\sigma}(\Omega)\cap L^\infty(\Omega)$, $u \in U$ be given.

Then there exists a unique solution to the equation $e(y,u) = 0$, i.e. there
exist
$(v^m, \varphi^m, \mu^m)_{m=1}^M \in Y$ such that
$(v^m,\varphi^m,\mu^m)$
is the unique solution to
\eqref{eq:TD:chns1_solenoidal_init}--\eqref{eq:TD:chns3_solenoidal} for
$m=1,\ldots,M$.
Moreover there holds
\begin{align*}
  \|(v^m)_{m=1}^M&\|_{H^{1}(\Omega)^n}
  +\|(\varphi^m)_{m=1}^M\|_{H^2(\Omega)}
  +\|(\mu^m)_{m=1}^M\|_{H^{2}(\Omega)}\\
  \leq&
  C_1\left(v^0\right)
  C_2\left(\|u_I\|_{H^1(\Omega)},
  \| (B_Vu_V^m)_{m=1}^M\|_{L^{2}(\Omega)^n},
  \| (B_Bu_B^m)_{m=1}^M\|_{H^{\frac{1}{2}}(\partial\Omega)^n}
  \right).
\end{align*}

Further $e(y,u) $ is Fréchet-differentiable with respect to $y$, and $e_y(y,u)
\in \mathcal{L}(Y_0,Z)$ has a bounded inverse. Thus Newton's method can be
applied for finding the unique solution of \eqref{eq:TD:eyu_eq_0} for given $u$.
\end{theorem}
\begin{proof}
The existence and stability of the solution for each time instance follows
directly from Theorem \ref{thm:TD:exSol_initStep} and Theorem \ref{thm:TD:exSol_twoStep}.

The equation $e(y,u) = 0$ is of block diagonal form
with nonlinear entries on the diagonal. Thus solving\eqref{eq:TD:eyu_eq_0}
reduces to solving each time instance with given data from the previous time instance.
As argued in Theorem \ref{thm:TD:exSol_initStep} and Theorem
\ref{thm:TD:exSol_twoStep} these nonlinear equations can be solved by Newton's
method. Applying this argument for all time instances we obtain that
$e_y(y,u) \in \mathcal{L}(Y_0,Z)$ has a bounded inverse.
\end{proof}

\begin{lemma}\label{lem:e_J_diff}
The functional $e:Y_0\times U \to Z$ is continuously differentiable with respect
to $y$ and $u$. Furthermore the equation $e(y_0,u) = 0$  for each $u$ admits a
unique solution $y(u)$, and $e_y(y_0,u)$ is continuously invertible.

The functional $J(y_0,u)$ is continuously differentiable with respect to $y_0$
and $u$.
\end{lemma}

Based on Lemma \ref{lem:e_J_diff} we introduce the reduced functional
$\hat J(u) := J(y_0(u),u)$ and state the following theorem.

\begin{theorem}[Existence of an optimal control]\label{thm:TD:exMin}
There exists at least one solution to \ref{prob:TD:Ptau}, i.e. at least one
optimal control.
\end{theorem}
\begin{proof}
Since $\hat J$ is bounded from below, there exists a minimizing sequence
$u_l$ with $\hat J(u_l) \to \hat J^\star$ and
$\hat J^\star := \inf_u \hat J(u)$.

Since $\hat J$ is radially unbounded, there exists a closed ball $V\subset U$,
bounded, convex and closed such that $u_l\subset V$ and thus there exists
a weakly convergent subsequence, in the following again denoted by $(u_l)$.
Since closed convex sets are weakly closed, $(u_l) \rightharpoonup u_\star \in V$
holds.
Let $y_l = (v_l,\varphi_l,\mu_l)$ denote the unique solution of
\eqref{eq:TD:chns1_solenoidal_init}--\eqref{eq:TD:chns3_solenoidal} for $u_l$.
Then $y_l \rightharpoonup y_\star \in Y$, with $y_\star = y_\star(u_\star)$, and
$(u_\star,y_\star)$ solves
\eqref{eq:TD:chns1_solenoidal_init}-\eqref{eq:TD:chns3_solenoidal}.
This can be shown as in
\cite[Thm. 6]{GarckeHinzeKahle_CHNS_AGG_linearStableTimeDisc}.

It remains to show, that $J(u_\star) = J^\star$.
Since $y_l \rightharpoonup y_\star$ especially
$\varphi_l^M\rightharpoonup\varphi^M_\star$ and thus by the lower weak semi
continuity of norms
together with an embedding argument for $W_u$
we have
\begin{align*}
  \hat J(u_\star) \leq \lim\inf \hat J(u_l) = \hat J^\star.
\end{align*}
Thus $u_\star$ is an optimal control.
\end{proof}

We next derive first order optimality conditions in the abstract setting.
We introduce an adjoint state $p\in Z^\star$ and the Lagrangian as
\begin{align*}
  L(y,p,u) := J(y,u) - \left<p,e(y,u)\right>_{Z^{\star},Z}.
\end{align*}

By  Lagrangian calculus we then obtain the following first order
optimality conditions.

\begin{theorem}[First order optimality conditions in abstract setting]
\label{thm:TD:OPT}
Let $u\in U$, $y\in Y$ be an optimal solution to \ref{prob:TD:Ptau}. Then
there exists an adjoint state $p\in Z^\star$ and the triple $(u,y,p)$ fulfills
the following first order optimality conditions:
\begin{align}
  e(y,u)&= 0 \in Z,\label{eq:OPT:primal}\\
  (e_y(y,u))^\star  p &= J_y(y,u) \in Y_0^\star,\label{eq:OPT:dual}\\
  \left<J_u(y,u) + (e_u(y,u))^\star p,w-u\right>_{U^\star,U} &= 0 \quad \forall w \in 
  U\label{eq:OPT:opt}.
\end{align}
\end{theorem}
\begin{proof}
From Lemma \ref{lem:e_J_diff} we have that $e$ and $J$ fulfill the assumptions
of \cite[Cor. 1.3]{HinzePinnauUlbrichUlbrich}, which in turn asserts the claim.
\end{proof}

To state the first order optimality system we introduce Lagrange multiplier $p
\in Z^\star$, $p = (p_v^m,p_\varphi^m,p_\mu^m)_{m=1}^M \in H_{0,\sigma}^M \times
H^1(\Omega)^M \times H^1(\Omega)^M$ and define the Lagrangian
\begin{align*}
  L:U
  &\times (H_{0,\sigma})^M
  \times ( H^1(\Omega) \cap L^\infty(\Omega) )^M
  \times W^{1,3}(\Omega)^M\\
  &\times (H_{0,\sigma})^M
  \times H^1(\Omega)^M
  \times H^1(\Omega)^M
  \to \mathbb{R}
\end{align*}
as
\begin{align*}
  &L(u,v_0^m,\varphi^m,\mu^m,p_v^m,p_\varphi^m,p_\mu^m) :=
  \frac{1}{2}
  \|\varphi^M-\varphi_d\|_{L^2(\Omega)}^2 \\
  &+ \frac{\alpha}{2}
  \left(
  \alpha_V \|u_V\|^2_{L^2(0,T;\mathbb{R}^{u_v})}
  +\alpha_B \|u_B\|^2_{L^2(0,T;\mathbb{R}^{u_b})}
  +\alpha_I\left(\int_\Omega \frac{\delta}{2}|\nabla u_I|^2
  + \frac{1}{\delta}W_u(u_I)\dx\right)
  \right)\\
  &-\sum_{m=2}^M
  \left[
  \frac{1}{\tau}\left(  \frac{\rho^{m-1}+\rho^{m-2}}{2}
  (v_0^m + \widetilde{B_Bu_B^m})
  -\rho^{m-2}(v_0^{m-1} +  \widetilde{B_Bu_B^{m-1}})
  ,p_v^{m} \right)
  \right.\\
  &+a(\rho^{m-1}(v_0^{m-1} +  \widetilde{B_Bu_B^{m-1}}) +J^{m-1},
  (v_0^m + \widetilde{B_Bu_B^m}),p_v^m)\\
 & +(2\eta^{m-1}D(v_0^m + \widetilde{B_Bu_B^m}),Dp_v^m) \\
  & \left.-(\mu^m\nabla \varphi^{m-1},p_v^m)
  -(\rho^{m-1}K,p_v^m)
  - (B_Vu_V^m,p_v^m)
  \vphantom{\frac{1}{2\tau}}\right]\\
  &-\sum_{m=2}^M
  \left[
  \frac{1}{\tau}(\varphi^m-\varphi^{m-1},p_\varphi^m)
  + ((v_0^m + \widetilde{B_Bu_B^m})\nabla  \varphi^{m-1},p_\varphi^m)
  + (b\nabla \mu^m,\nabla p_\varphi^m)
  \right]\\
  &-\sum_{m=2}^M
  \left[
  \sigma\epsilon(\nabla \varphi^m,\nabla p_\mu^m) - (\mu^m,p_\mu^m) +
  \frac{\sigma}{\epsilon}
  \left(
  W'_+(\varphi^m) + W'_-(\varphi^{m-1}),p_\mu^m
  \right)
  \right]\\
  &-\left[\frac{1}{\tau}\left(  \frac{\rho^{1}+\rho^{0}}{2}
  (v^1+\widetilde{B_Bu_B^1})-\rho^{0}v^{0},p_v^{1}
  \right)+a(\rho^{1}v^{0}+J^{1}, (v^1+\widetilde{B_Bu_B^1}),p_v^1)
   \right.\\
  & \left.+(2\eta^{1}D (v^1+\widetilde{B_Bu_B^1}),Dp_v^1)-(\mu^1\nabla u_I,p_v^1)
  -(\rho^{0}K,p_v^1)
  - (B_Vu_V^1,p_v^1)
  \vphantom{\frac{1}{2\tau}}\right]\\
  & - \left[ \frac{1}{\tau}(\varphi^1-u_I,p_\varphi^1)
  + (v^0\nabla  u_I,p_\varphi^1)
  + (b\nabla \mu^1,\nabla p_\varphi^1)
  \right]\\
  &- \left[ \sigma\epsilon(\nabla \varphi^1,\nabla p_\mu^1) - (\mu^1,p_\mu^1) +
  \frac{\sigma}{\epsilon}
  \left(
  W'_+(\varphi^1) + W'_-(u_I),p_\mu^1
  \right)
  \right].
\end{align*}
Here again $\rho^m := \rho(\varphi^m)$, $\eta^m := \eta(\varphi^m)$ and
especially $\rho^0 := \rho(u_1)$, $\eta^0 := \eta(u_1)$.
In the following we write $v^m := v^m_0 + \widetilde{B_Bu_B^m}$.

The optimality system is now given by $(DL(x),\tilde x -x) \geq 0$, where
$x$ abbreviates all arguments of $L$ and $\tilde x$ denotes an admissible direction.
For all components of $x$ except $u_I$ it even holds $(DL(x),\tilde x) = 0$ 
since there no further constraints apply, while $U_I$ is a convex subset of $H^1(\Omega)\cap
L^\infty(\Omega)$.

\medskip
\noindent\textbf{Derivative with respect to the velocity}\\
The derivative with respect to $v_0^m$ for $m=2,\ldots,M$ into a direction
$\tilde v \in H_{0,\sigma}$ is given by
\begin{equation}
  \begin{aligned}
    (D_{v^m}L(\ldots,&v^m,\ldots),\tilde v)=\\
    & -
    \frac{1}{\tau}\left(
    (\frac{\rho^{m-1}+\rho^{m-2}}{2}\tilde v,p_v^{m})
    -(\rho^{m-1}\tilde v,p_v^{m+1})
    \right)\\
    &-a(\rho^m\tilde v,v_0^{m+1} + \widetilde{B_Bu_B^{m+1}},p_v^{m+1})\\
    &-a(\rho^{m-1}(v_0^{m-1} + \widetilde{B_Bu_B^{m-1}})+J^{m-1},\tilde v,p_v^m)\\
    &-(2\eta^{m-1}D\tilde v,Dp_v^m)
    -(\tilde v\nabla \varphi^{m-1},p_\varphi^m)= 0.
  \end{aligned}
  \label{eq:TD:adj:Dv}
\end{equation}
For $m=1$ we get
\begin{equation}
  \begin{aligned}
    (D_{v^1}L(&\ldots,v^1,\ldots),\tilde v)=\\
    &\frac{1}{\tau}(\rho^0\tilde v,p_v^2)
    -a(\rho^1\tilde v,v_0^2 + \widetilde{B_Bu_B^2},p_v^2)\\
    &-\frac{1}{2\tau}((\rho^1+\rho^0){\tilde v},p_v^1)
    -a(\rho^1(v_0^0 + \widetilde{B_Bu_B^0})+J^1,{\tilde v},p_v^1)
    -(2\eta^1D\tilde v,Dp_v^1)=  0.
  \end{aligned}
  \label{eq:TD:adj:Dv_init}
\end{equation}
Note that for notational convenience here we introduce
artificial variables $v_0^{M+1}$, $p_v^{M+1}$, $u_B^{M+1}$
and set them to $v_0^{M+1} \equiv p_v^{M+1} \equiv 0$, $u_B^{M+1} = 0$.

\begin{remark}
Note that we derive the adjoint system in the solenoidal setting. Introducing
a variable $p$ for the pressure in the primal equation leads to an additional
adjoint variable $p_p$ for the adjoint pressure and to an additional term
$(-div{\tilde v},p_p)$.
\end{remark}

\medskip

\noindent \textbf{Derivative with respect to the chemical potential}\\
The derivative with respect to the chemical potential  for
$m=2,\dots,M$ in a direction $\tilde \mu \in W^{1,3}(\Omega)$ is
\begin{equation}
  \begin{aligned}
    (D_{\mu^m}&L(\ldots,\mu^m,\ldots),\tilde \mu) =\\
    &-a(J^m_{\mu^{m}}\tilde \mu,v^{m+1},p_v^{m+1})
    + (\tilde \mu\nabla \varphi^{m-1},p_v^{m})
    -(b\nabla \tilde \mu,\nabla p_\varphi^m)
    +(\tilde \mu,p_\mu^m)=0.
  \end{aligned}
  \label{eq:TD:adj:Dmu}
\end{equation}
For $m=1$ the equations is
\begin{equation}
  \begin{aligned}
    (D_{\mu^1}&L(\ldots,\mu^1,\ldots),\tilde \mu) =\\
    & -a(J^1_\mu{\tilde \mu},v^2,p_v^2)
    -a(J^1_\mu{\tilde \mu},v^1,p_v^1)
     + (\tilde\mu\nabla u_I,p_v^1) 
     -(b\nabla \tilde \mu,\nabla p_\varphi^1)
     +(\tilde \mu,p_\mu^1) = 0.
  \end{aligned}
  \label{eq:TD:adj:Dmu_init}
\end{equation}
Here for $m=1,\ldots,M$ we abbreviate $J^m_\mu{\tilde \mu} = -\rho_\delta b
\nabla {\tilde \mu}$, and for notational convenience we introduce
artificial variables $v^{M+1} = v_0^{M+1} + \widetilde{B_Bu_B^{M+1}}$,
$p_v^{M+1}$, and set them to $v^{M+1} \equiv p_v^{M+1} \equiv 0$.

The above also contains the boundary condition
\begin{align*}
  \nabla p^m_\varphi\cdot \nu_\Omega = 0\quad
  m=1,\ldots,M,
\end{align*}
in weak form, which for smooth $p_\varphi^m$ follows from integration by parts.

\medskip

\noindent \textbf{Derivative with respect to the phase field}\\
The derivative with respect to  the phase field $\varphi^m$ in a direction
$\tilde \varphi \in H^1(\Omega)\cap L^\infty(\Omega)$ is for  $m=2,\dots,M$
\begin{equation}
\begin{aligned}
  (D_{\varphi^m}&L(\ldots,\varphi^m,\ldots),\tilde \varphi) =\\
  &\delta_{mM}(\varphi^m-\varphi_d,{\tilde\varphi})
  -\frac{1}{\tau}
  \left(\rho'
  \frac{v^{m+1}p_v^{m+1}+v^{m+2}p_v^{m+2}}{2},
  \tilde \varphi
  \right)
  +
  \frac{1}{\tau}\left(
  \rho'v^{m+1}p_v^{m+2},\tilde \varphi
  \right)
  \\
  &-a(\rho'\tilde \varphi v^{m},v^{m+1},p_v^{m+1})
  -(2\eta'\tilde\varphi Dv^{m+1},Dp_v^{m+1})\\
  & +(\mu^{m+1}\nabla\tilde\varphi,p_v^{m+1})
  +(\rho'\tilde\varphi K,p_v^{m+1})\\
  &-\frac{1}{\tau}
  \left(
  (\tilde\varphi,p_\varphi^m) - (\tilde\varphi,p_\varphi^{m+1})
  \right)
  -(v^{m+1}\nabla \tilde \varphi,p_\varphi^{m+1})\\
  &-\sigma\epsilon(\nabla \tilde \varphi,\nabla p_\mu^m)
  -\frac{\sigma}{\epsilon}(W''_+(\varphi^m)\tilde \varphi,p_\mu^m)
  -\frac{\sigma}{\epsilon}(W''_-(\varphi^m)\tilde\varphi,p_\mu^{m+1})=0,
\end{aligned}
\label{eq:TD:adj:Dphi}
\end{equation}
where $\delta_{mM}$ denotes the Kronecker delta.
For $m=1$ we get
\begin{equation}
\begin{aligned}
  (D_{\varphi^1}&L(\ldots,\varphi^1,\ldots),\tilde \varphi) =\\
  &
  -\frac{1}{\tau}(\frac{\rho^\prime}{2}{\tilde \varphi},v^2p_v^2 )
  -a(\rho^\prime{\tilde \varphi}v^1,v^2,p_v^2)
  -a(\rho^\prime{\tilde \varphi}v^0,v^1,p_v^1) \\
  &-(2\eta^\prime{\tilde \varphi}Dv^2,Dp_v^2)
  -(2\eta^\prime{\tilde \varphi}Dv^1,Dp_v^1)
  -(\mu^2\nabla {\tilde \varphi} p_v^2)
  -(\rho^\prime{\tilde \varphi}K,p_v^2)\\
  & +\frac{1}{\tau}(\tilde \varphi,p_\varphi^2)
  -(v^2\nabla {\tilde \varphi},p_\varphi^2)
  -\frac{\sigma}{\epsilon}
  (W_-^{\prime\prime}(\varphi^1)\tilde  \varphi,p_\mu^2)\\
  %
  %now derivative of first time step
  &-\frac{1}{\tau}( \frac{\rho^\prime{\tilde \varphi}}{2}v^1,p_v^1)
  -\frac{1}{\tau}({\tilde \varphi},p_\varphi^1)
  -\sigma\epsilon(\nabla {\tilde \varphi},\nabla p_\mu^1)
  -\frac{\sigma}{\epsilon}(W_+^{\prime\prime}(\varphi^1){\tilde
  \varphi},p_\mu^1) = 0.
\end{aligned}
\label{eq:TD:adj:Dphi_init}
\end{equation}
Here for notational convenience we introduce
artificial variables $v^{M+1} = v_0^{M+1} + \widetilde{B_Bu_B^{M+1}}$,
$v^{M+2} = v_0^{M+2} + \widetilde{B_Bu_B^{M+2}}$,
$p_v^{M+1}$, $p_v^{M+2}$, and set them to zero.

The above also contains the boundary condition
\begin{align*}
  \nabla p_\mu^m\cdot \nu_\Omega = 0\quad
  m=1,\ldots,M,
\end{align*}
in weak form, which for smooth $p_\mu^m$ follows from integration by parts.

%%%%%%%%%%%%%%%%%%%%%%%%%%%%%%%%%%%%%%%%%%%%%%%%%%%%%%%%%%%%%%%%%%

\medskip

\noindent\textbf{Derivative with respect to the control}\\
Finally we calculate the derivative with respect to the control for the three
parts of the control space.

For a test direction $w \in U_V$ we have
\begin{align*}
  (D_{u_V}&L(u,\ldots),w) =
  \alpha\alpha_V \int_I(u_V,w)_{\mathbb{R}^{u_v}}\dt
  +\sum_{m=1}^M(B_V w^m,p_v^m)_{L^2(\Omega)} =0,
\end{align*}
and thus the optimality condition is
\begin{align}
  \alpha\tau \alpha_Vu_V^m + B_V^*p_v^m = 0
  \in \mathbb{R}^{u_v} \quad m=1,\ldots,M
  \label{eq:TD:opt_UV}
\end{align}
Here $B_V^\star p_v^m$ is defined as
\begin{align*}
  B_V^\star p_v^m :=
  ((f_l,p_v^m)_{L^2(\Omega)^n})_{l=1}^{u_v}.
\end{align*}

\bigskip

Concerning the derivative with respect to $u_B$ we have for a test function
$w \in U_B$
\begin{equation}
\begin{aligned}
  (D_{u_B}&L(u,\ldots), w) =
  \alpha\alpha_B \int_I(u_B,w)_{\mathbb{R}^{u_b}}\dt
  -\tau^{-1}\left(\frac{\rho^1+\rho^0}{2} \widetilde{B_B
  w^1},p_v^1\right) \\
  &-a(\rho^1v^0 + J^1,\widetilde{B_B w^1},p_v^1)
  -2(\eta^1D\widetilde{B_B w^1},Dp_v^1)\\
  &-\sum_{m=2}^{M} \left[
  \tau^{-1}\left(\frac{\rho^{m-1}+\rho^{m-2}}{2}
  \widetilde{B_B w^m},p_v^m\right)
  -\left(\rho^{m-2}\widetilde{B_B w^{m-1}},p_v^{m}\right)
  \right.\\
  &\left.
  +a(\rho^{m-1}v^{m-1} + J^{m-1},\widetilde{B_B w^m},p_v^m)
  +a(\rho^{m-1}\widetilde{B_B w^{m-1}},v^{m},p_v^{m}) \right.\\
  &\left. \vphantom{\frac{\rho^m}{2}}
  +2(\eta^{m-1}D\widetilde{B_B w^m},Dp_v^m)
  +(\widetilde{B_B w^m}\nabla \varphi^{m-1},p_\varphi^m)
  \right]
  =0.
\end{aligned}
\label{eq:TD:opt_UB}
\end{equation}

For smooth solutions we use the derivative with respect the velocity, the no-flux boundary
condition for $v^m$ as well as for $\mu^m$ and  integration by parts to observe
\begin{align*}
  (D_{u_B}L(u,\ldots), w)& =
  \alpha\alpha_B \int_I(u_B,w)_{\mathbb{R}^{u_b}}\dt\\
  &- \sum_{m=2}^M \int_{\partial\Omega} 2 \eta^{m-1}Dp_v^m\cdot \nu_\Omega
  B_Bw^m\ds
  -\int_{\partial\Omega} 2 \eta^{1}Dp_v^1\cdot \nu_\Omega
  B_Bw^1\ds
\end{align*}
and thus the optimality condition in a strong formulation is
\begin{equation}
\begin{aligned}
  \alpha\alpha_B \tau u_B^m
  - \left(
  (2\eta^{m-1}Dp_v^m \cdot \nu_\Omega,g^l)
  _{H^{-1/2}(\partial\Omega),H^{1/2}(\partial\Omega)}\right)_{l=1}^{u_b} = 0 \in
  \mathbb{R}^{u_b}
  \quad
  \forall m=2,\ldots,M,\\
   \alpha\alpha_B \tau u_B^1
  - (2\eta^{1}Dp_v^1 \cdot \nu_\Omega,g^l)
  _{H^{-1/2}(\partial\Omega),H^{1/2}(\partial\Omega)} = 0 \in
  \mathbb{R}^{u_b} .
\end{aligned}
\end{equation}

\bigskip

The derivative with respect to the initial condition $u_I$ in a
direction $w-u_I \in U_I$ is
\begin{equation}
\begin{aligned}
  (D_{u_I}&L(u,\ldots),w-u_I)_{U_I^\star,U_I} =
  \frac{\alpha}{2}\alpha_I
  \left(
  \epsilon(\nabla u_I,\nabla (w-u_I))
  +  \epsilon^{-1}\int_\Omega W_u^\prime(u_I)(w-u_I) \dx
  \right)\\
  &-\frac{1}{2\tau}\left(\rho'(w-u_I) v^2,p_v^2\right)
  +\frac{1}{\tau}\left(\rho'(w-u_I) v^1,p_v^2\right)\\
  &
  -\frac{1}{2\tau}\left(\rho'(w-u_I) v^1,p_v^1\right)
  +\frac{1}{\tau}\left(\rho'(w-u_I) v^0,p_v^1\right)\\
  &+(\mu^1\nabla (w-u_I),p_v^1)
  +(\rho'(w-u_I) K,p_v^1)\\
  &+\frac{1}{\tau}((w-u_I),p_\varphi^1)
  -(v^0\nabla (w-u_I),p_\varphi^1)
  -\frac{\sigma}{\epsilon}(W^{\prime\prime}_-(u_I)(w-u_I),p_\mu^1) \geq 0.
\end{aligned}
\label{eq:TD:opt_UI}
\end{equation}
We note that $u_I \in H^1(\Omega) \cap L^\infty(\Omega)$ and thus that there exists no gradient
representation for $D_{u_I}L$. This is reflected later in our numerical approach.

\begin{remark}\label{rm:TD:uVB_discrete__uI_Linfty}
  From \eqref{eq:TD:opt_UV} we see, that in fact $u_V$ has a discrete structure
  with respect to time, namely it is piecewise constant over time intervals, as
  the adjoint variable $p_v$ is. The same holds for $u_B$.
\end{remark}

%____________________________________________________________________________________________________________________________________________
%_____________________________________________________________________________________________________________________________
% ___________________Fully_discrete_setting____________________________________________________________________________
%_________________________________________________________________________________________________________________________
%________________________________________________________________________________________________________________________

\section{The fully discrete setting}\label{sec:FullyDiscrete}
%section identifier: FD
We next use finite elements to discretize the optimal control problem
\ref{prob:TD:Ptau} in space.
For this we use finite elements on locally adapted meshes.
At time instance $t_m$, $m=1,\ldots,M$ we use a  quasi-uniform,
triangulation of $\overline \Omega$ with $NT_m$ triangles denoted by
$\mathcal T_m = \{T_i\}_{i=1}^{NT_m}$ fulfilling $ \overline\Omega=
\bigcup_{i=1}^{NT_m} \overline T_i $.

On $\mathcal T_{m}$ we define the following finite element spaces:
\begin{align*}
  \mathcal{V}_m^{1}
  =& \{v\in C(\mathcal{T}_{m}) \, |
  \, v|_T \in \mathcal  P^1(T)\, \forall T\in  \mathcal{T}_{m}\},\\
  \mathcal{V}_m^{2} =& \{v\in C(\mathcal{T}_{m})^n \, |
  \, v|_T \in \mathcal  P^2(T)^n\, \forall T\in  \mathcal{T}_{m} \},
\end{align*}
where $\mathcal P^l(S)$ denotes the space of polynomials up to order $l$ defined
on $S$.
We note that by construction $\mathcal V^1_m \subset
W^{1,\infty}(\mathcal T_m)$ and
$\mathcal V^2_m \subset
W^{1,\infty}(\mathcal T_m)^n$
holds.
We introduce the discrete analog to the space
$H_{\sigma}(\Omega)$:
\begin{align*}
  H_{\sigma,m} &:=
  \{ v\in \mathcal V^2_m\,|\, (\mbox{div}v,q) =  0\,
  \forall q\in \mathcal V^1_m \cap L^2(\Omega)\},
\end{align*}
and
\begin{align*}
  H_{0,\sigma,m} :=
  \{ v\in H_{\sigma,m}\,|\, \gamma(v) = 0 \}.
\end{align*}

We further  introduce a linear $H^1$-stable projection operator
${P}^{m} : H^1(\Omega) \to \mathcal{V}^1_m$
satisfying
\begin{align*}
  \|P^mv\|_{L^p(\Omega)} \leq C\|v\|_{L^p(\Omega)},
  \mbox{  and  }
  \|\nabla P^{m}v \|_{L^r(\Omega)} \leq C\|\nabla v\|_{L^r(\Omega)},
\end{align*}
for $v \in H^1(\Omega)$ with $r \in [1,2]$ and $p\in [1,6]$ if $n=3$, and $p\in [1,\infty)$ if
$n=2$ and
\begin{align*}
  \|P^m v - v\|_{H^1(\Omega)} \to 0
\end{align*}
for $h\to 0$ for $v\in H^2(\Omega)$
Typically examples are the Cl\'ement operator  or,
by restricting the preimage to $C(\overline \Omega) \cap H^1(\Omega)$, the Lagrangian interpolation operator.

Using these spaces we state the discrete counterpart of
\eqref{eq:TD:chns1_solenoidal_init}--\eqref{eq:TD:chns3_solenoidal}: \\

Let $u\in U$ and $v_0 \in H_{\sigma} \cap L^\infty(\Omega)^n$ be given.\\
\noindent \textit{Initialization for $m=1$:}\\
Set $\varphi_h^0 := u_I$, $v^0 := v_0$.
Find
$v^{1}_h \in  H_{\sigma,1}$,
$\gamma(v^1_h) = \Pi^1(B_Bu_B^1)$,
$\varphi^{1}_h\in \mathcal V^1_1$,
$\mu^{1}_h \in \mathcal V^1_1$
such that for all
$w \in  H_{0,\sigma,1}$,
$\Psi \in \mathcal V^1_1$,
$\Phi \in  \mathcal V^1_1$
it holds:
\begin{align}
  \tau^{-1}\left(\frac{1}{2}(\rho_h^{1}+\rho_h^0)v^{1}_h-\rho_h^{0}v^{0},w\right)
  + a(\rho_h^{1}v^{0}+J_h^{1},v_h^{1},w)\nonumber\\
  +(2\eta_h^{1}Dv_h^{1},D w)-(\mu^{1}_h\nabla\varphi_h^{0}+\rho_h^{0} g,w)
  -(B_Vu_V^1,w)
  &= 0,\label{eq:FD:chns1_solenoidal_init}\\
  \frac{1}{\tau}(\varphi^{1}_h-{P}^{1}\varphi_h^{0},\Psi)
  +(b\nabla \mu^{1}_h,\nabla \Psi)
  +(v^{0}\nabla \varphi_h^{0}, \Psi)
  &=0,\label{eq:FD:chns2_solenoidal_init}\\
  \sigma\epsilon(\nabla \varphi^{1}_h,\nabla \Phi)
  +\frac{\sigma}{\epsilon}(W_+^\prime(\varphi^{1}_h)+W^\prime_-({P}^{1}\varphi_h^{0}),\Phi)-(\mu^{1}_h,\Phi)
  &=0, \label{eq:FD:chns3_solenoidal_init}
\end{align}
where $J^1 := - \rho_\delta b \nabla \mu_h^1$.

\medskip

\noindent \textit{Two-step scheme for $m>1$:}\\
Given
$\varphi_h^{m-2}\in \mathcal{V}^1_{m-2}$,
$\varphi_h^{m-1}\in \mathcal{V}^1_{m-1}$,
$\mu_h^{m-1}\in \mathcal{V}^1_{m-1}$,
$v_h^{m-1}\in  H_{\sigma,m-1}$,
find
$v^{m}_h \in  H_{\sigma,m}$,
$\gamma(v^m_h) = \Pi^m(B_Bu_B^m)$,
$\varphi^{m}_h\in \mathcal V^1_m$,
$\mu^{m}_h \in \mathcal V^1_m$
such that for all
$w \in  H_{0,\sigma,m}$,
$\Psi \in \mathcal V^1_m$,
$\Phi \in  \mathcal V^1_m$
it holds:
\begin{align}
  \tau^{-1}
  \left(
  \frac{1}{2}(\rho_h^{m-1}+\rho_h^{m-2})v^{m}_h-\rho_h^{m-2}v_h^{m-1},w
  \right)
  + a(\rho_h^{m-1}v_h^{m-1}+J_h^{m-1},v_h^{m},w)\nonumber\\
  +(2\eta_h^{m-1}Dv^{m}_h,D w)-(\mu^{m}_h\nabla\varphi_h^{m-1}+\rho_h^{m-1} g,w)
  -(B_Vu_V^m,w)
  &= 0,\label{eq:FD:chns1_solenoidal}\\
  \frac{1}{\tau}(\varphi^{m}_h-{P}^{m}\varphi_h^{m-1},\Psi)
  +(b\nabla \mu^{m}_h,\nabla \Psi)
  +(v^{m}_h\nabla \varphi_h^{m-1}, \Psi)
  &=0,\label{eq:FD:chns2_solenoidal}\\
  %1
  %
  %
  \sigma\epsilon(\nabla \varphi^{m}_h,\nabla \Phi)
  +\frac{\sigma}{\epsilon}(W_+^\prime(\varphi^{m}_h)+W^\prime_-({P}^{m}\varphi_h^{m-1}),\Phi)-(\mu^{m}_h,\Phi)
  &=0, \label{eq:FD:chns3_solenoidal}
\end{align}
where  $J_h^{m-1} := - \rho_\delta b\nabla \mu_h^{m-1}$.

% Here $\Pi^m$ for $m=1,\ldots,M$ denotes the $L^2(\partial\Omega)$ projection onto the trace space of
% $\mathcal V_m^2$ and is used to incorporate the boundary data.
We introduce 
\begin{align*}
  \mathcal V^2_{m,b} := \{ v|_{\partial\Omega} \,|\, v \in \mathcal{V}^2_m,\, 
  \int_{\partial\Omega} v|_{\partial\Omega} \cdot \nu_\Omega\, \ds = 0 \}
\end{align*}
and define $\Pi^m$ for $m=1,\ldots,M$ as the $L^2(\partial\Omega)$ projection onto the trace space
of $\mathcal V^2_{m,b}$. This projection is used to incorporate the boundary data and fulfills
$\|\Pi^mg-g\|_{L^2(\partial\Omega)} \to 0$ for all $g\in H^{1/2}(\partial\Omega)$ with
$\int_{\partial\Omega} g\cdot\nu_\Omega\ds = 0$.

We require bounds with respect to $W^{1,p}(\Omega)$-norms for the solution of
\eqref{eq:FD:chns1_solenoidal_init}--\eqref{eq:FD:chns3_solenoidal} and prepare these with the
following lemmas.

\begin{lemma}
\label{lem:FD:neumannLaplace_supremum}
For all  $1<p<\infty$ there exists a continuous function $C(p)$, such that
\begin{align*}
  \|\nabla u\|_{L^p(\Omega)} 
  \leq C(p) \sup_{
  \stackrel{\eta\in L^q(\Omega), \eta  \neq  0}{(\eta,1)=0}} \frac{(\nabla u,\nabla \eta)}{\|\nabla \eta\|_{L^q(\Omega)}
  },
\end{align*}
where $\frac{1}{p}+\frac{1}{q}=1$.
Further, from the generalized Poincar\'e inequality, 
\cite[Thm. 8.16]{Alt_lineareFunktionalAnalysis},
 we obtain
$ \|\eta\|_{W^{1,q}(\Omega)} \leq C\|\nabla \eta\|_{L^q(\Omega)}$ 
and thus
\begin{align*}
    \|\nabla u\|_{L^p(\Omega)} 
  \leq C(p) \sup_{
  \stackrel{\eta\in L^q(\Omega), \eta  \neq  0}{(\eta,1)=0}} \frac{(\nabla u,\nabla
  \eta)}{\|\eta\|_{W^{1,q}(\Omega)} }.
\end{align*}

\end{lemma}
\begin{proof}
The proof follows  as in
\cite[Lem. 1.1]{2005_BarrettGarckeNuernberg_surfaceDiffusion_in_stressed_solid}
and uses $L^p$-stability for $u$ shown in 
\cite[Thm. 1.2]{2010_Geng_Shen__Neumann_problem_Helmholtz}.
\end{proof}

\begin{lemma}
\label{lem:FD:NeumannLaplace_Stability}
For $v \in W^{1,p}(\Omega)$ let $Q_hv \in \mathcal V_m^1$  be defined by
\begin{align}
  (\nabla Q_hv,\nabla w) &= (\nabla v, \nabla w) \quad \forall w \in \mathcal V_m^1,\\
  \int_\Omega Q_hv \dx &= \int_\Omega v \dx.
  \label{eq:FD:def_Qh}
\end{align}
Let $1<p<\infty$.
Then it holds
\begin{align}
  \|Q_hv\|_{W^{1,p}(\Omega)} \leq C(p)\|v\|_{W^{1,p}(\Omega)}.
  \label{eq:FD:NeumLapl_Stab}
\end{align}
\end{lemma} 
\begin{proof}
The proof follows the lines of \cite[Ch. 8]{BrennerScott}.
However, from the fact that the boundary data is of Neumann type new difficulties arise and we refer
to \cite{2005_BarrettGarckeNuernberg_surfaceDiffusion_in_stressed_solid} and
\cite{NuernbergTucker__FEA_nanostructurePatterning} how to deal with these issues.
\end{proof}

\begin{lemma}
\label{lem:FD:NeuLap_sup_disc}
Let $u_h\in \mathcal V_m^1 \subset W^{1,q}(\Omega)$.
Then it holds
\begin{align*}
  \|\nabla u_h\|_{L^p(\Omega)} 
  \leq C(p)\sup_{
  \stackrel{\eta_h \in V^1_m, \eta_h \neq 0}{(\eta,1) = 0}
  } 
  \frac{(\nabla u_h,\nabla \eta_h)}{\|\eta_h\|_{W^{1,q}(\Omega)}},
\end{align*}
where $\frac{1}{p}+\frac{1}{q}=1$.
\end{lemma}
\begin{proof}
Directly follows by combining Lemma~\ref{lem:FD:neumannLaplace_supremum}, the definition of $Q_hv$
in \eqref{eq:FD:def_Qh} and the stability estimate \eqref{eq:FD:NeumLapl_Stab}, compare 
\cite[Thm. 2.3]{NuernbergTucker__FEA_nanostructurePatterning}.
\end{proof}

\begin{theorem}\label{thm:FD:exSol_oneStep}
For given $v_0 \in H^1(\Omega)^n \cap L^\infty(\Omega)^n$, $u\in U$ there exist
$v^1_h \in H_{\sigma,1}$, $\gamma(v^1_h) = \Pi^1(B_Bu_B^1)$,
$\varphi^1_h \in \mathcal V^1_1$ and $\mu^1_h\in \mathcal V^1_1$
solving \eqref{eq:FD:chns1_solenoidal_init}--\eqref{eq:FD:chns3_solenoidal_init}.
It further holds
\begin{align*}
     &\|\mu^1_h\|_{W^{1,3}(\Omega)} 
   + \|\varphi^1_h\|_{W^{1,4}(\Omega)}
   + \|v^1_h\|_{H^1(\Omega)}\\
   &\leq
    C_1(v^0)
  C_2\left(\|u_I\|_{H^1(\Omega)},
  \|B_Vu_V^1\|_{L^2(\Omega)^n},
  \|B_Bu_B^1\|_{H^{\frac{1}{2}}(\partial\Omega)^n}\right),
\end{align*}
and Newton's method can be used to find the unique solution to
\eqref{eq:FD:chns1_solenoidal_init}--\eqref{eq:FD:chns3_solenoidal_init}.
\end{theorem}
\begin{proof}
For \eqref{eq:FD:chns2_solenoidal_init}--\eqref{eq:FD:chns3_solenoidal_init} the
existence of a unique solution and the applicability of Newton's method  follows
from \cite{HintermuellerHinzeTber}. Also the stability in $H^1$ is proven there.

To obtain the estimates of higher regularity we use Lemma~\ref{lem:FD:NeuLap_sup_disc}.
It holds
\eqref{eq:FD:chns2_solenoidal_init}
\begin{equation}
\begin{aligned}
 C \| \mu_h^1&\|_{W^{1,3}(\Omega)} 
 \leq \|\nabla \mu_h^1\|_{L^{3}(\Omega)}  + \|\mu_h^1\|_{L^3(\Omega)}\\
   \leq& \|\mu_h^1\|_{L^3(\Omega)}+C\sup_{
   \stackrel{v_h \in  V^1_m, (v_h,1) = 0}
   {\|v_h\|_{W^{1,\frac{3}{2}}(\Omega)} = 1}}
  (\nabla \mu_h^1,\nabla  v_h) \\
  \leq& C\|\mu_h^1\|_{H^1(\Omega)} +C\sup
  \left(\left|\frac{1}{\tau}(\varphi^1_h-P^1\varphi_h^0,v_h)\right| + \left|(v^0\nabla  \varphi_h^0,v_h)\right|\right)\\
  \leq &C \|\mu_h^1\|_{H^1(\Omega)} +C\sup \left(\|\varphi_h^1 -
  P^1\varphi_h^0\|_{L^2(\Omega)} \|v_h\|_{L^{2}(\Omega)}
  +
  \|v^0\nabla\varphi_h^0\|_{L^2(\Omega)^n}\|v_h\|_{L^2(\Omega)}
  \right)\\
  \leq & C \|\mu_h^1\|_{H^1(\Omega)} +C\|\varphi_h^1 -  P^1\varphi_h^0\|_{L^2(\Omega)}
  + C\|v^0\nabla \varphi_h^0\|_{L^2(\Omega)^n}\\
  \leq&C \left( \|\mu_h^1\|_{H^1(\Omega)}+
  \|\varphi_h^1\|_{H^1(\Omega)} + \|\varphi_h^0\|_{H^1(\Omega)}
  +\|v^0\|_{L^\infty(\Omega)^n}\|\varphi_h^0\|_{H^1(\Omega)}\right)
\end{aligned}
\label{eq:FD:higherBound_muh}
\end{equation}
which, together with the already known bound for $\|\mu_h^1\|_{H^1(\Omega)}$ states the bound on
$\mu_h^1$ in $W^{1,3}(\Omega)$.
Note the continuous embedding $W^{1,\frac{3}{2}}(\Omega) \hookrightarrow L^2(\Omega)$
used for $v_h$. %NOTE: higher bound then $W^{1,6}$ not possible as \varphi^0 \in H^1 only

For $\varphi_h^1$ we argue similarly and estimate 
\begin{align*}
 C&\|\varphi_h^1\|_{W^{1,4}(\Omega)}\\
  \leq & \|\varphi_h^1\|_{L^4(\Omega)} 
  +C \sup_{\stackrel{v_h \in
  W^{1,\frac{4}{3}}(\Omega),(v_h,1) = 0}{\|v\|_{W^{1,\frac{4}{3}}(\Omega)} = 1}}
  \left(  (\nabla \varphi_h^1,\nabla  v_h)
  \right)\\
  \leq &C\|\varphi_h^1\|_{H^1(\Omega)}  + C\sup \left(\left|(\mu_h^1,v_h)\right| +
  \left|\frac{\sigma}{\epsilon}(W_+^\prime(\varphi_h^1)+W_-^\prime(P^1\varphi_h^0),v_h)\right|\right)\\
  \leq & C\|\varphi_h^1\|_{H^1(\Omega)} + C\|\mu_h^1\|_{L^2(\Omega)}
  + C \sup \left[ (1+|\varphi_h^1|^{q-1},|v_h|) 
  + (1+|P^1\varphi_h^0|^{q-1},|v_h|)
  \right]\\
  \leq & C\|\varphi_h^1\|_{H^1(\Omega)} + C\|\mu_h^1\|_{L^2(\Omega)}\\%%
  &+ C \left( \|1+|\varphi_h^1|^{q-1}\|_{L^2(\Omega)}
  + \|1+|P^1\varphi_h^0|^{q-1}\|_{L^2(\Omega)}
  \right)\sup \|v_h\|_{L^2(\Omega)}\\
  \leq& C\|\varphi_h^1\|_{H^1(\Omega)} + C\|\mu_h^1\|_{L^2(\Omega)}
  +C \left( 1 + \|\varphi_h^1\|_{H^1(\Omega)}  +   \|\varphi_h^0\|_{H^1(\Omega)}
  \right).
\end{align*}
%info: Projektion nicht in W^{1,1} stabil!, Ern Guermond thm 3.21
We note the continuous embeddings $W^{1,\frac{4}{3}}(\Omega) \hookrightarrow L^2(\Omega)$ and 
$H^1(\Omega) \hookrightarrow L^6(\Omega)$.

The existence of a unique solution for \eqref{eq:FD:chns1_solenoidal_init} and
stability for $v^1_h$ then follows from standard arguments for the Oseen
equation, since we use an LBB-stable finite element pair.
\end{proof}

\begin{theorem}\label{thm:FD:exSol_twoStep}
For given $u\in U$, $\varphi^{m-2}\in \mathcal{V}^1_{m-2}$,
$\varphi^{m-1}\in \mathcal{V}^1_{m-1}$,
$\mu^{m-1}\in \mathcal{V}^1_{m-1}$,
$v^{m-1}\in  H_{\sigma,m-1}$,
and for all $m=2,\ldots,M$ there exist $v^m_h \in
H_{\sigma,m}$, $\gamma(v^m_h) = \Pi^m(B_Bu_B^m)$,
$\varphi^m_h \in \mathcal V^1_m$ and $\mu^m_h\in \mathcal V^1_m$ solving
\eqref{eq:FD:chns1_solenoidal}--\eqref{eq:FD:chns3_solenoidal}.

It further holds
\begin{align*}
  &\|\mu^m_h\|_{W^{1,3}(\Omega)} 
  + \|\varphi^m_h\|_{W^{1,4}(\Omega)}
  +\|v^m_h\|_{H^1(\Omega)^n}
  \\
  &\leq 
   C\left(\|v_h^{m-1}\|_{H^1(\Omega)^n},
   \|\mu_h^{m-1}\|_{W^{1,3}(\Omega)},\|\varphi_h^{m-1}\|_{W^{1,4}(\Omega)},\right.\\
   &\left.
   \|B_Vu_V^m\|_{L^2(\Omega)^n},
   \|B_Bu_B^m\|_{H^{\frac{1}{2}}(\partial\Omega)^n}\right),
\end{align*}
and the constant depends polynomially on its arguments.
\end{theorem}
\begin{proof}
In \cite{GarckeHinzeKahle_CHNS_AGG_linearStableTimeDisc} the existence of unique
solutions to \eqref{eq:FD:chns1_solenoidal}--\eqref{eq:FD:chns3_solenoidal} together with bounds in
$H^1(\Omega)$ on the solution is shown for the case $B_Vu_V^m=0$, $B_Bu_B^m=0$, 
using \cite[Lem. II 1.4]{Temam_NavierStokes_old}.
The volume force $B_Vu_V^m$ is  given data that enters the proof in a
straightforward manner. 
The boundary data $B_Bu_B^m$ can be incorporated by
investigating a shifted system as in Theorem \ref{thm:TD:exSol_twoStep}.

The estimates of higher regularity follow as in Theorem~\ref{thm:FD:exSol_oneStep}.
There the bound for $\mu_h^1$ relies on $L^\infty(\Omega)$ regularity of
$v^0$, that is not available here. Instead in \eqref{eq:FD:higherBound_muh} we can use a
$L^6(\Omega)$ bound for $v_h^m$ that directly follows from the $H^1(\Omega)$ bound by Sobolov embedding, together with the $L^3(\Omega)$
bound for $\nabla \varphi_h^{m-1}$.
\end{proof}

\begin{theorem}\label{thm:FD:exSol}
Let $v^0 \in H^1(\Omega)^n \cap L^\infty(\Omega)$, $u\in U$ be given.
Then there exist sequences $(v^m)_{m=1}^M \in (H_{\sigma,m})_{m=1}^M$,
$(\varphi^m)_{m=1}^M,(\mu^m)_{m=1}^M \in \mathcal (\mathcal V^1_m)_{m=1}^M$,
such that $(v^m,\varphi^m,\mu^m)$
is the unique solution to
\eqref{eq:FD:chns1_solenoidal_init}--\eqref{eq:FD:chns3_solenoidal} for
$m=1,\ldots,M$.
Moreover there holds
\begin{align*}
  & \|(v_h^m)_{m=1}^M  \|_{H^1(\Omega)}
   +\| (\mu^m_h)_{m=1}^M \|_{W^{1,3}(\Omega)}
   + \| (\varphi^m_h)_{m=1}^M  \|_{W^{1,4}(\Omega)}
   \\
   &\leq C_1(v^0)
  C_2\left(
  \|u_I\|_{H^1(\Omega)},
  \|(B_Vu_V^m)_{m=1}^M\|_{L^2(\Omega)^n},
  \|(B_Bu_B^m)_{m=1}^M\|_{H^{\frac{1}{2}}(\partial\Omega)^n}
  \right).
\end{align*}
Here the constants $C_1,C_2$ depend polynomially on their arguments.
\end{theorem}
\begin{proof}
The existence of the solution for each time instance follows directly from
Theorem \ref{thm:FD:exSol_oneStep} and Theorem \ref{thm:FD:exSol_twoStep}.
The stability estimate follows from iteratively applying the stability estimates
from  Theorem \ref{thm:FD:exSol_oneStep};
\end{proof}

\begin{remark}
  The bounds with respect to higher norms are required in Section~\ref{ssec:lim}
  for the limit process $h \to 0$.
\end{remark}

To derive first order necessary optimality conditions we argue as in the case of
the time discrete optimization problem and show that Newton's method can be used
for solving the primal equation
\eqref{eq:FD:chns1_solenoidal_init}--\eqref{eq:FD:chns3_solenoidal}  on each
time instance.

\begin{theorem}\label{thm:FD:newt}
Newton's method can be used for finding the unique solution to
\eqref{eq:FD:chns1_solenoidal_init}--\eqref{eq:FD:chns3_solenoidal}
on each time instance.
\end{theorem}
\begin{proof}
For $m=1$ this is argued in Theorem \ref{thm:FD:exSol_oneStep}.
For $m>1$ we abbreviate
equation \eqref{eq:FD:chns1_solenoidal}--\eqref{eq:FD:chns3_solenoidal} by
$F((v^m_h,\varphi^m_h,\mu^m_h),(w,\Phi,\Psi)) = 0$.
Then $F$ is Fr\'echet differentiable, since all terms are linear beside the
term $W^{\prime}_+$ which is differentiable by Assumption \ref{ass:psi_C2}.
The derivative in the direction
$(\delta v, \delta \varphi,\delta \mu) \in H_{0,\sigma,m} \times
\mathcal V^1_m \times \mathcal V^1_m$
is given by
\begin{align*}
  \langle G&(v^{m}_h,\varphi^{m}_h,\mu^{m}_h)
  (\delta v,\delta \varphi, \delta \mu),(w,\Phi,\Psi) \rangle:=\\
  &\frac{1}{\tau}\left(\frac{\rho^{m-1}+\rho^{m-2}}{2}\delta v,w\right) +
  a(\rho^{m-1}v^{m-1}+J^{m-1},\delta v, w)\\
  & + (\eta^{m-1}D\delta v, Dw) -
  (\delta \mu\nabla \varphi^{m-1},w) \\
  &+\frac{1}{\tau}(\delta\varphi,\Psi) + (b\nabla \delta \mu,\nabla \Psi)
  + (\delta v\nabla \varphi^{m-1},\Psi)\\
  &+\sigma\epsilon(\nabla \delta \varphi,\nabla \Phi)+
  \frac{\sigma}{\epsilon}(W^{\prime\prime}_+(\varphi^{m}_h)\delta
  \varphi,\Phi) -(\delta \mu,\Phi).
\end{align*}
The existence of a solution
$(\delta v, \delta \varphi,\delta \mu)$ can be shown following
\cite[Thm. 2]{GarckeHinzeKahle_CHNS_AGG_linearStableTimeDisc}, using Brouwer's
fixpoint theorem. The boundedness of $(\delta v, \delta \varphi,\delta \mu)$
follows from the same proof.
\end{proof}

We next introduce the fully discrete analog to problem \eqref{prob:TD:Ptau}.
\begin{equation}\label{prob:FD:Ph}\tag{$\mathcal{P}_h$}
  \begin{aligned}
    \min_{u\in U} J( (\varphi_h^m)_{m=1}^M,u)& =
    \frac{1}{2}\|\varphi_h^M-\varphi_d\|^2_{L^2(\Omega)}\\
    & +\frac{\alpha}{2}
    \left(
    \alpha_I \left(
    \int_\Omega \frac{\delta}{2} |\nabla u_I|^2 +
    \delta^{-1}W_u(u_I)\dx \right)
    \right.\\
    &\left.\vphantom{\int_\Omega}
    + \alpha_V \|u_V\|^2_{L^2(0,T;\mathbb{R}^{u_v})}
    + \alpha_B \|u_B\|^2_{L^2(0,T;\mathbb{R}^{u_b})}
    \right)\\
    \mbox{s.t. }
    \eqref{eq:FD:chns1_solenoidal_init} - \eqref{eq:FD:chns3_solenoidal}.
  \end{aligned}
\end{equation}
We stress, that we do not discretize the control for the initial value.
However for a practical implementation we need a discrete description for
$u_I$.
This will be discussed after deriving the optimality conditions, see Section
\ref{sec:num}.

\begin{theorem}[Existence of an optimal discrete control] \label{thm:FD:exMin}
There exists at least one optimal control to \ref{prob:FD:Ph}.
\end{theorem}
\begin{proof}
The claim follows from standard arguments, compare Theorem \ref{thm:TD:exMin}.
\end{proof}

\bigskip

We next state the fully discrete counterpart of the first order optimality
conditions from Section \ref{sec:TD}.

For this we  introduce adjoint
variables $ (p_{v,h}^m)_{m=1}^M \in (H_{0,\sigma,m})_{m=1}^M$,
$(p_{\varphi,h}^m)_{m=1}^M \in (\mathcal V^1_m)_{m=1}^M$, and
$(p_{\mu,h}^m)_{m=1}^M \in (\mathcal V^m)_{m=1}^M$.
For convenience in the following we often  write $v_h^m :=
v_{0,h}^m + \widetilde{B_Bu_B^m}$.

By the same Lagrangian calculus as in Section \ref{sec:TD} we obtain the
following fully discrete optimality system.

\medskip
\noindent\textbf{Derivative with respect to the velocity}\\
The derivative with respect to $v_{0,h}^m$ for $m=2,\ldots,M$ into a direction
$\tilde v \in \mathcal V^2_m$ is given by
\begin{equation}
  \begin{aligned}
    (D_{v_h^m}L(\ldots,&v_h^m,\ldots),\tilde v)=\\
    & -
    \frac{1}{\tau}\left(
    (\frac{\rho_h^{m-1}+\rho_h^{m-2}}{2}\tilde v,p_{v,h}^{m})
    -(\rho_h^{m-1}\tilde v,p_{v,h}^{m+1})
    \right)\\
    &-a(\rho_h^m\tilde v,v_h^{m+1},p_{v,h}^{m+1})
    -a(\rho_h^{m-1}v_h^{m-1}+J_h^{m-1},\tilde v,p_{v,h}^m)\\
    &-(2\eta_h^{m-1}D\tilde v,Dp_{v,h}^m)
    -(\tilde v\nabla \varphi_h^{m-1},p_{\varphi,h}^m)= 0.
  \end{aligned}
  \label{eq:FD:adj:Dv}
\end{equation}
For $m=1$ we get
\begin{equation}
  \begin{aligned}
    (D_{v_h^1}L(&\ldots,v_h^1,\ldots),\tilde v)=\\
    &-\frac{1}{2\tau}((\rho_h^1+\rho^0){\tilde
    v},p_{v,h}^1)
    + \frac{1}{\tau}(\rho^0\tilde v,p_{v,h}^2)
    -a(\rho_h^1\tilde  v,v_h^2,p_{v,h}^2)\\
    &
    -a(\rho_h^1 v_h^0+J_h^1,{\tilde v},p_{v,h}^1)
    -(2\eta_h^1D\tilde v,Dp_{v,h}^1)=  0.
  \end{aligned}
  \label{eq:FD:adj:Dv_init}
\end{equation}
Note that for notational convenience here we introduce
artificial variables $v_h^{M+1}$, $p_{v,h}^{M+1}$,
and set them to $v_h^{M+1} \equiv p_{v,h}^{M+1} \equiv 0$.

\medskip

\noindent \textbf{Derivative with respect to the chemical potential}\\
The derivative with respect to the chemical potential  for
$m=2,\dots,M$ in a direction $\tilde \mu \in \mathcal V^1_m$ is
\begin{equation}
  \begin{aligned}
    (D_{\mu_h^m}&L(\ldots,\mu_h^m,\ldots),\tilde \mu) =\\
    &-a(J_{\mu}\tilde \mu,v_h^{m+1},p_{v,h}^{m+1})
    + (\tilde \mu\nabla \varphi_h^{m-1},p_{v,h}^{m})
    -(b\nabla \tilde \mu,\nabla p_{\varphi,h}^m)
    +(\tilde \mu,p_{\mu,h}^m)=0.
  \end{aligned}
  \label{eq:FD:adj:Dmu}
\end{equation}
For $m=1$ the equations is
\begin{equation}
  \begin{aligned}
    (D_{\mu^1}&L(\ldots,\mu^1,\ldots),\tilde \mu) =\\
    & -a(J_\mu{\tilde \mu},v_h^2,p_{v,h}^2)
    -a(J^1_\mu{\tilde \mu},v_h^1,p_{v,h}^1) 
    + (\tilde \mu\nabla u_I,p_{v,h}^1)\\
     &-(b\nabla \tilde \mu,\nabla p_{\varphi,h}^1) +(\tilde \mu,p_{\mu,h}^1) = 0.
  \end{aligned}
  \label{eq:FD:adj:Dmu_init}
\end{equation}
Here for $m=1,\ldots,M$ we abbreviate $J^m_\mu{\tilde \mu} = -\rho_\delta b
\nabla {\tilde \mu}$ and for notational convenience we introduce
artificial variables $v_h^{M+1}$ and
$p_{v,h}^{M+1}$, and set them to $v_h^{M+1} \equiv p_{v,h}^{M+1} \equiv 0$.

\medskip

\noindent \textbf{Derivative with respect to the phase field}\\
The derivative with respect to  the phase field $\varphi_h^m$ in a direction
$\tilde \varphi \in \mathcal{V}^1_m$ is for  $m=2,\dots,M$
\begin{equation}   
\begin{aligned}
  (D_{\varphi_h^m}&L(\ldots,\varphi_h^m,\ldots),\tilde \varphi) =\\
  &\delta_{mM}(\varphi_h^m-\varphi_d,{\tilde\varphi})
  -\frac{1}{\tau}
  \left(\rho'
  \frac{v_h^{m+1}p_{v,h}^{m+1}+v_h^{m+2}p_{v,h}^{m+2}}{2},
  \tilde \varphi
  \right)
  +
  \frac{1}{\tau}\left(
  \rho'v_h^{m+1}p_{v,h}^{m+2},\tilde \varphi
  \right)
  \\
  &-a(\rho'\tilde \varphi v_h^{m},v_h^{m+1},p_{v,h}^{m+1})
  -(2\eta'\tilde\varphi Dv_h^{m+1},Dp_{v,h}^{m+1})\\
  & +(\mu_h^{m+1}\nabla\tilde\varphi,p_{v,h}^{m+1})
  +(\rho'\tilde\varphi g,p_{v,h}^{m+1})\\
  &-\frac{1}{\tau}
  \left(
  (\tilde\varphi,p_{\varphi,h}^m) - (P^{m+1}\tilde\varphi,p_{\varphi,h}^{m+1})
  \right)
  -(v_h^{m+1}\nabla \tilde \varphi,p_{\varphi,h}^{m+1})\\
  &-\sigma\epsilon(\nabla \tilde \varphi,\nabla p_{\mu,h}^m)
  -\frac{\sigma}{\epsilon}(W''_+(\varphi_h^m)\tilde \varphi,p_{\mu,h}^m)
  -\frac{\sigma}{\epsilon}(W''_-(P^{m+1}\varphi_h^m)P^{m+1}\tilde\varphi,p_{\mu,h}^{m+1})=0.
\end{aligned}
  \label{eq:FD:adj:Dphi}
\end{equation}
Here $\delta_{mM}$ denotes the Kronecker delta of $m$ and $M$.
For $m=1$ we get
\begin{equation} 
\begin{aligned}
  (D_{\varphi_h^1}&L(\ldots,\varphi_h^1,\ldots),\tilde \varphi) =\\
  &
  -\frac{1}{\tau}(\frac{\rho^\prime}{2}{\tilde \varphi},v_h^2p_{v,h}^2 )
  -a(\rho^\prime{\tilde \varphi}v_h^1,v_h^2,p_{v,h}^2)
  -a(\rho^\prime{\tilde \varphi}v^0,v_h^1,p_{v,h}^1)\\
  &-(2\eta^\prime{\tilde \varphi}Dv_h^2,Dp_{v,h}^2)
  -(2\eta^\prime{\tilde \varphi}Dv_h^1,Dp_{v,h}^1)
  -(\mu_h^2\nabla {\tilde \varphi} p_{v,h}^2)
  -(\rho^\prime{\tilde \varphi}g,p_{v,h}^2)\\
  & +\frac{1}{\tau}(P^2\tilde \varphi,p_{\varphi,h}^2)
  -(v_h^2\nabla {\tilde \varphi},p_{\varphi,h}^2)
  -\frac{\sigma}{\epsilon}
  (W_-^{\prime\prime}(P^{2}\varphi_h^1)P^2\tilde  \varphi,p_{\mu,h}^2)\\
  %
  %now derivative of first time step
  &-\frac{1}{\tau}( \frac{\rho^\prime{\tilde \varphi}}{2}v_h^1,p_{v,h}^1)
  -\frac{1}{\tau}({\tilde \varphi},p_{\varphi,h}^1)
  -\sigma\epsilon(\nabla {\tilde \varphi},\nabla p_{\mu,h}^1)
  -\frac{\sigma}{\epsilon}(W_+^{\prime\prime}(\varphi_h^1){\tilde
  \varphi},p_{\mu,h}^1) = 0.
\end{aligned}
\label{eq:FD:adj:Dphi_init}
\end{equation}
Here for notational convenience we introduce
artificial variables $v_h^{M+1}$, $v_h^{M+2}$, $p_{v,h}^{M+1}$, and
$p_{v,h}^{M+2}$, and set them to zero.

\begin{remark}
We note that the projection operator $P^m$ enters 
\eqref{eq:FD:adj:Dphi}--\eqref{eq:FD:adj:Dphi_init} 
acting on the test function $\tilde \varphi$.
\end{remark}

%%%%%%%%%%%%%%%%%%%%%%%%%%%%%%%%%%%%%%%%%%%%%%%%%%%%%%%%%%%%%%%%%%

\medskip

\noindent\textbf{Derivative with respect to the control}\\
Finally we calculate the derivative with respect to the control for the three
parts of the control space.

For a test direction $w \in U_V$ we have
\begin{align*}
  (D_{u_V}&L(u,\ldots),w) =
  \alpha\alpha_V \int_I(u_V,w)_{\mathbb{R}^{u_v}}\dt
  +\sum_{m=1}^M(B_V w^m,p_{v,h}^m)_{L^2(\Omega)} =0,
\end{align*}
and thus the optimality condition is
\begin{align}
  \alpha\tau \alpha_Vu_V^m + B_V^*p_{v,h}^m = 0
  \in \mathbb{R}^{u_v} \quad m=1,\ldots,M.
  \label{eq:FD:opt_UV}
\end{align}
Here $B_V^\star p_v^m$ is defined as
\begin{align*}
  B_V^\star p_v^m :=
  ((f_l,p_{v,h}^m)_{L^2(\Omega)^n})_{l=1}^{u_v}.
\end{align*}

\bigskip

Concerning the derivative with respect to $u_B$ we have for a test function
$w \in U_B$
\begin{equation}
\begin{aligned}
  (D_{u_B}&L(u,\ldots), w) =
  \alpha\alpha_B \int_I(u_B,w)_{\mathbb{R}^{u_b}}\dt
  -\tau^{-1}\left(\frac{\rho_h^1+\rho^0}{2} \widetilde{B_B
  w^1},p_{v,h}^1\right) \\
  &-a(\rho_h^1v^0 + J_h^1,\widetilde{B_B w^1},p_{v,h}^1)
  -2(\eta_h^1D\widetilde{B_B w^1},Dp_{v,h}^1)\\
  &-\sum_{m=2}^{M} \left[
  \tau^{-1}\left(\frac{\rho_h^{m-1}+\rho_h^{m-2}}{2}
  \widetilde{B_B w^m},p_{v,h}^m\right)
  -\left(\rho_h^{m-2}\widetilde{B_B w^{m-1}},p_{v,h}^{m}\right)
  \right.\\
  &\left.
  +a(\rho_h^{m-1}v_h^{m-1} + J_h^{m-1},\widetilde{B_B w^m},p_{v,h}^m)
  +a(\rho_h^{m-1}\widetilde{B_B w^{m-1}},v_h^{m},p_{v,h}^{m}) \right.\\
  &\left. \vphantom{\frac{\rho^m}{2}}
  +2(\eta_h^{m-1}D\widetilde{B_B w^m},Dp_{v,h}^m)
  +(\widetilde{B_B w^m}\nabla \varphi_h^{m-1},p_{\varphi,h}^m)
  \right]\\
 & =: \alpha\alpha_B \int_I(u_B,w)_{\mathbb{R}^{u_b}}\dt
  + F_h(w) = 0.
\end{aligned}
\label{eq:FD:opt_UB}
\end{equation} 
Here $F_h(w)$ abbreviates the action of the discrete normal derivative of $p_{v,h}$, see e.g.
\cite{HinzePinnauUlbrichUlbrich}.

\bigskip

The derivative with respect to the initial condition $u_I$ in any
direction $w-u_I \in U_I$ is
\begin{equation}
\begin{aligned}
  (D_{u_I}&L(u,\ldots),w-u_I)_{U_I^\star,U_I} =
  \frac{\alpha}{2}\alpha_I
  \left(
  \epsilon(\nabla u_I,\nabla (w-u_I))
  +  \epsilon^{-1}\int_\Omega W_u^\prime(u_I)(w-u_I) \dx
  \right)\\
  &-\frac{1}{2\tau}\left(\rho'(w-u_I) v_h^2,p_{v,h}^2\right)
  +\frac{1}{\tau}\left(\rho'(w-u_I) v_h^1,p_{v,h}^2\right)\\
  &-\frac{1}{2\tau}\left(\rho'(w-u_I) v_h^1,p_{v,h}^1\right)
  +\frac{1}{\tau}\left(\rho'(w-u_I) v^0,p_{v,h}^1\right)\\
  &
  +(\mu_h^1\nabla (w-u_I),p_{v,h}^1)
  +(\rho'(w-u_I) K,p_{v,h}^1)\\
  &+\frac{1}{\tau}((w-u_I),p_{\varphi,h}^1)
  -(v^0\nabla (w-u_I),p_{\varphi,h}^1)
  -\frac{\sigma}{\epsilon}(W^{\prime\prime}_-(u_I))(w-u_I),p_{\mu,h}^1) \geq 0,
\end{aligned}
\label{eq:FD:opt_UI}
\end{equation}
and this inequality holds for all $w\in U_I$.

\begin{remark}
We
use the finite
element space $\mathcal V^1_1$
for the representation of $u_I$.
\end{remark}

%-----------------------------------------------------------------
%-----------OPT_h__Ende------------------------------------------
%-----------------------------------------------------------------

\subsection{The limit $h \to 0$}
\label{ssec:lim}
%section identifier: lim
We next investigate the limit $h\to 0$ for problem \ref{prob:FD:Ph}.
Let $u^\star,\varphi^\star$ denote a solution to \ref{prob:TD:Ptau}
and $u_h,\varphi_h$ denote a solution to
\ref{prob:FD:Ph}.
Since $u_h,\varphi_h$ is a minimizer for $J$ in the discrete setting, we have
$J(u_h,\varphi_h)\leq J(P_hu^\star,P_h\varphi^\star) \leq CJ(u^\star,\varphi^\star) 
=Cj$, where $P_h$ denotes any $H^1$-stable projection onto the discrete spaces.
Thus
\begin{equation}
  \begin{aligned}
    \frac{1}{2}&  \|\varphi^M_h - \varphi_d\|^2
    + \frac{\alpha}{2}
    \left(
    \alpha_I
    \int_\Omega \frac{\epsilon}{2} |\nabla u_{I,h}|^2 +
    \epsilon^{-1}W_u(u_{I,h})\dx
    \right.\\
    &\left.\vphantom{\int_\Omega}
    + \alpha_V \|u_{B,h}\|^2_{L^2(0,T;\mathbb{R}^{u_v})}
    + \alpha_B \|u_{V,h}\|^2_{L^2(0,T;\mathbb{R}^{u_b})}
    \right) \leq Cj.
  \end{aligned}
  \label{eq:lim:J_bounded}
\end{equation} 
Note that the mean value of $u_{I,h}$ is fixed and thus by Poincar\'es
inequality we have $\|u_{I,h}\|_{H^1(\Omega)} \leq C(1+\|\nabla u_{I,h}\|)$.

Thus from \eqref{eq:lim:J_bounded} we obtain the following bounds uniform in
$h$:
\begin{align*}
  \|u_{I,h}\|_{H^1(\Omega)}
+\|u_{B,h}\|_{L^2(0,T;\mathbb{R}^{u_v})}
+   \|u_{V,h}\|_{L^2(0,T;\mathbb{R}^{u_b})}\leq C.
\end{align*}
Using Theorem \ref{thm:FD:exSol} we further get the bounds
\begin{align*}
  \|(v_h^m)_{m=1}^M\|_{H^1(\Omega)^n}
  +\|(\mu_h^m)_{m=1}^M\|_{W^{1,3}(\Omega)}
  +  \|(\varphi_h^m)_{m=1}^M\|_{W^{1,4}(\Omega)}
  \leq C.
\end{align*}
Using Lax-Milgram's theorem and the above bounds  
 we further obtain bounds
\begin{align*}
  \| (p_{v,h}^m)_{m=1}^M\|_{H^1(\Omega)^n}
  +  \| (p_{\varphi,h}^m)_{m=1}^M\|_{H^1(\Omega)}
  + \| (p_{\mu,h}^m)_{m=1}^M\|_{H^1(\Omega)}  \leq C
\end{align*}
for the adjoint variables.

Now there exist $u_I^\star \in H^1(\Omega)$,
$u_V^\star \in L^2(0,T;\mathbb{R}^{u_v})$,
$u_B^\star \in L^2(0,T;\mathbb{R}^{u_b})$
such that
\begin{align*}
  u_{I,h} \rightharpoonup u_I^\star, \quad
  u_{V,h} \rightharpoonup u_V^\star, \quad
  u_{B,h} \rightharpoonup u_B^\star.
\end{align*}
There further exist
$(v^{m,\star})_{m=1}^M \in (H^1(\Omega)^n)^M$,
$(\varphi^{m,\star})_{m=1}^M \in W^{1,4}(\Omega)^M$, and
$(\mu^{m,\star})_{m=1}^M \in W^{1,3}(\Omega)^M$
such that
\begin{align*}
  v_h^m       \rightharpoonup v^{m,\star},\quad
  \varphi_h^m \rightharpoonup \varphi^{m,\star},\quad
  \mu_{h}^m   \rightharpoonup \mu^{m,\star} \quad \forall m=1,\ldots,M.
\end{align*}
And there further exist
$(p_v^{m,\star})_{m=1}^M \in H^1(\Omega)^M$,
$(p_\varphi^{m,\star})_{m=1}^M \in H^1(\Omega)^M$, and
$(p_\mu^{m,\star})_{m=1}^M \in H^1(\Omega)^M$
such that
\begin{align*}
  p_{v,h}^m       \rightharpoonup p_v^{m,\star},\quad
  p_{\varphi,h}^m \rightharpoonup p_\varphi^{m,\star},\quad
  p_{\mu,h}^m   \rightharpoonup p_\mu^{m,\star} \quad \forall m=1,\ldots,M.
\end{align*}

Now let us proceed to the limit in the fully discrete optimality system.
To this end we will especially show the following strong convergence results
\begin{align*}
  \varphi_h^m &\to \varphi^m \quad \mbox{in } H^1(\Omega),\\
  \mu_h^m &\to \mu^m \quad \mbox{in } W^{1,3}(\Omega),\\
  v_h^m &\to v^m \quad \mbox{in } H_\sigma(\Omega),\\
  u_{I,h} & \to u_I \quad \mbox{in } H^1(\Omega), 
\end{align*} 
for $m=1,\ldots,M$.

\medskip
 
\noindent\textbf{The limit $h\to 0$ in the primal equation}\\
The convergence of \eqref{eq:FD:chns3_solenoidal} to
\eqref{eq:TD:chns3_solenoidal} and of \eqref{eq:FD:chns3_solenoidal_init} to
\eqref{eq:TD:chns3_solenoidal_init} follows directly from the proposed weak
convergences together with the strong convergence $\varphi^m_h \to \varphi^m$ in
$L^{\infty}$ obtained by compact Sobolev embedding.
 To obtain strong convergence in $H^1(\Omega)$ we argue as 
in the proof of Theorem~\ref{thm:FD:exSol_oneStep}. 

Let $B:H^1(\Omega)\times H^1(\Omega) \to \mathbb R$ denote the coercive bilinear form
$B(u,v) = \sigma\epsilon (\nabla u,\nabla v) + (u,v)$ and let $Q_h\varphi^1 \in \mathcal V^1_1$
denote the projection of $\varphi^1$ onto $\mathcal V_1^1$ with respect to $B$ fulfilling
$\|Q_h\varphi^1-\varphi^1\|_{H^1(\Omega)} \to 0$ for $h\to 0$, since $\varphi^1 \in H^2(\Omega)$.

Then it holds
\begin{align*}
 & \|\varphi^1_h - \varphi^1\|_{H^1(\Omega)} 
  \leq \|\varphi^1_h - Q_h\varphi^1\|_{H^1(\Omega)}+\|Q_h\varphi^1 - \varphi^1\|_{H^1(\Omega)},
  \end{align*}
  and
  \begin{align*}
 &C\|\varphi^1_h - Q_h\varphi^1\|^2_{H^1(\Omega)}\\
&\leq B(\varphi^1_h - Q_h\varphi^1,\varphi^1_h - Q_h\varphi^1)
= B(\varphi^1_h-\varphi^1,\varphi^1_h - Q_h\varphi^1)\\
\leq & | (\mu_h^1-\mu^1,\varphi_h^1-Q_h\varphi^1)|  
+ \|\varphi_h^1-\varphi^1\|_{L^2(\Omega)}\|\varphi_h^1-Q_h\varphi^1\|_{L^2(\Omega)}\\
 &+ \frac{\sigma}{\epsilon} |(W^\prime_+(\varphi_h^1)-W^\prime_+(\varphi^1),\varphi_h^1-Q_h\varphi^1)| 
 + \frac{\sigma}{\epsilon} |
 (W^\prime_-(P^1\varphi_h^0)-W^\prime_-(\varphi^0),\varphi_h^1-Q_h\varphi^1)|\\
 \leq & \|\mu_h^1-\mu^1\|_{L^2(\Omega)}\|\varphi_h^1-Q_h\varphi^1\|_{L^2(\Omega)} 
 +  \|\varphi_h^1-\varphi^1\|_{L^{2}(\Omega)} \|\varphi_h^1-Q_h\varphi^1\|_{L^{2}(\Omega)}\\
& +
\frac{\sigma}{\epsilon}\|W_+^\prime(\varphi_h^1)-W_+^\prime(\varphi^1)\|_{L^{5/3}(\Omega)}\|\varphi_h^1-Q_h\varphi^1\|_{L^{5/2}(\Omega)}\\
 &+
 \frac{\sigma}{\epsilon}\|W_-^\prime(P^1\varphi_h^0)-W_-^\prime(\varphi^0)\|_{L^{5/3}(\Omega)}\|\varphi_h^1-Q_h\varphi^1\|_{L^{5/2}(\Omega)}.
\end{align*}
Using Sobolev embedding $H^1(\Omega) \hookrightarrow L^{p}(\Omega)$, $p\leq 6$ and dividing by
$\|\varphi_h^1-Q_h\varphi^1\|_{H^1(\Omega)}$ the resulting differences tend to zero by 
compact Sobolev embedding, 
or by Lebesgue's generalized convergence theorem and Assumption~\ref{ass:psi_boundedPoly}.
The same arguments apply for the case $m>1$.

\medskip

The convergence of equation \eqref{eq:FD:chns2_solenoidal} to \eqref{eq:TD:chns2_solenoidal} and
\eqref{eq:FD:chns2_solenoidal_init} to \eqref{eq:TD:chns2_solenoidal_init} is
shown using the strong convergence $v_h^m \to v^m$ in $L^3(\Omega)$ together
with weak convergence $\nabla \varphi_h^{m-1} \rightharpoonup \nabla
\varphi^{m-1}$ in $L^2(\Omega)$ yielding weak convergence of the transport term
$v^m_h\nabla \varphi^{m-1}_h$ in $L^{6/5}$. For $m=1$ $v^0\nabla \varphi_h^0$  converges weakly in
$L^2(\Omega)$.
Further, strong convergence $\mu_h^m \to \mu^m$ in $H^1(\Omega)$ follows as above.

To show strong convergence in $W^{1,3}$ it is thus sufficient to show strong convergence for $\nabla
\mu_h^1 \to \nabla \mu^1$ in $L^3(\Omega)$. 
We define $Q_hv \in \mathcal V^1_1$ by 
\begin{align*}
(\nabla (Q_hv - v) , \nabla w_h)& = 0 \quad 
\forall w_h \in \mathcal V^1_1,\\  
(Q_hv,1)& = (v,1),
\end{align*}
satisfying $\|Q_hv\|_{W^{1,\frac{3}{2}}(\Omega)} \leq C\|v\|_{W^{1,\frac{3}{2}}(\Omega)}$,
Lemma~\ref{lem:FD:NeumannLaplace_Stability}.

We adapt the idea from Theorem~\ref{thm:FD:exSol_oneStep} and proceed
\begin{align*}
  &C\|\nabla \mu_h^1-\nabla\mu^1\|_{L^{3}(\Omega)}\\
  &\leq \sup_{\stackrel{v\in W^{1,\frac{3}{2}}(\Omega),(v,1)=0} {\|v\|_{W^{1,\frac{3}{2}}(\Omega)}=1}}
   (\nabla(\mu_h^1-\mu^1),\nabla v) = \sup\left[(\nabla (\mu_h^1-\mu^1),\nabla Q_hv) + (\nabla
   (\mu_h^1-\mu^1),\nabla (v-Q_hv) )\right] \\
  &\leq \sup\left[ (\nabla (\mu_h^1-\mu^1),\nabla Q_hv) + (\nabla\mu^1,\nabla (Q_hv-v)) \right]\\
  &\leq C\sup\left[ |\tau^{-1}(\varphi_h^1-\varphi^1,Q_hv)| 
  + |\tau^{-1}(P^1\varphi_h^0-\varphi^0,Q_hv)| 
  + |(v^0\nabla \varphi_h^0 -   v^0\nabla\varphi^0,Q_hv)\right.\\
  & \left. + |\tau^{-1}(\varphi^1-\varphi^0,Q_hv-v)|
  + |(v^0\nabla \varphi^0,Q_hv-v)| \right]\\
  & \leq C \left[
  \|\varphi_h^1-\varphi^1\|_{L^2(\Omega)}
  +  \|P^1\varphi^0_h-\varphi^0\|_{L^2(\Omega)}\right.\\
  &\left.+ \|\varphi^1-\varphi^0\|_{L^2(\Omega)}\sup \|Q_hv-v\|_{L^2(\Omega)}
  + \|v^0\nabla \varphi^0\|_{L^2(\Omega)}\sup \|Q_hv-v\|_{L^2(\Omega)}
  \right]\\
&  + C\sup |(v^0\nabla Q_hv,\varphi^0_h-\varphi^0)|\\
&\leq C \left[
  \|\varphi_h^1-\varphi^1\|_{L^2(\Omega)}
  +  \|P^1\varphi^0_h-\varphi^0\|_{L^2(\Omega)}\right.\\
  &\left.+\|\varphi^1-\varphi^0\|_{L^2(\Omega)}\sup \|Q_hv-v\|_{L^2(\Omega)}
  + \|v^0\nabla \varphi^0\|_{L^2(\Omega)}\sup \|Q_hv-v\|_{L^2(\Omega)}
  \right]\\
 & + C\|v^0\|_{L^\infty(\Omega)}\|\varphi^0_h-\varphi^0\|_{L^2(\Omega)}.
\end{align*}
Note that we used integration by parts to deal with the transport term.
From the H\"older and Sobolev inequalities it follows
\begin{align*}
  \|Q_hv-v\|_{L^2(\Omega)}^2 \leq \|Q_hv-v\|_{L^{\frac{3}{2}}(\Omega)}\|Q_hv-v\|_{L^{3}(\Omega)}.
\end{align*} 
The last term is bounded due to the fact, that $\|v\|_{W^{1,\frac{3}{2}}(\Omega)}\leq 1$ and 
$Q_h$ is stable in $W^{1,\frac{3}{2}}(\Omega)$. 
Since $\|Q_hv-v\|_{L^{\frac{3}{2}}(\Omega)} \leq Ch\|v\|_{W^{1,\frac{3}{2}}(\Omega)}$ we obtain
$\|Q_hv-v\|_{L^2(\Omega)} \to 0$ for $h\to 0$ and thus the strong convergence of $\nabla \mu_h$ in
$L^3(\Omega)$.
If $m>1$ we can use the strong convergence $\varphi_h^{m-1} \to \varphi^{m-1}$ in $H^1(\Omega)
\hookrightarrow L^6(\Omega)$ together with $\|v_h^m\|_{L^6(\Omega)} \leq C$ to treat the transport
term.
%info: we need this special Q to avoid the need of convcergence in W^{1,3/2}

\medskip

Next we consider the convergence of
\eqref{eq:FD:chns1_solenoidal} to \eqref{eq:TD:chns1_solenoidal} and
\eqref{eq:FD:chns1_solenoidal_init} to \eqref{eq:TD:chns1_solenoidal_init}.
Here the convergence $(\eta_h^{m-1}Dv_h^m : Dw) \to (\eta^{m-1}Dv^m : Dw)$
follows from the strong convergence $\varphi_h^{m-1} \to \varphi^{m-1}$ in
$L^\infty(\Omega)$ (by compact embedding $W^{1,4}(\Omega) \hookrightarrow
L^\infty(\Omega)$) and the weak convergence $Dv_h^m \rightharpoonup Dv^m$ in
$L^2(\Omega)$.
The convergence of the trilinear form is obtained by using the just shown strong
convergence $\nabla \mu_h^m \to \nabla \mu^m$ in $L^3(\Omega)$ together with the
weak convergence of $v_h^m \rightharpoonup v^m$ in $L^6(\Omega)$.

Let us finally show strong convergence $v_h^m \to v^m$ in
$H^1(\Omega)^n$.
Let $B:H_{\sigma}\times H_{\sigma} \to \mathbb R$ denote the coercive bilinear form 
$B(u,v) = (\eta_h^1Du:Dv) + (u,v)$.
The coercivity of $B$ follows from Korn's
inequality.
Let $w_h\in H_{\sigma,1}$ denote a sequence, such that
$\|w_h-v^1\|_{H^1(\Omega)^n} \to 0$ for $h\to 0$ and 
$w_h|_{\partial\Omega} \equiv v_h^1|_{\partial\Omega}$. 
The weak continuity of $\Pi^1$ ensures that $\Pi^1(v_h^1) \to v^1|_{\partial\Omega}$ and thus such
sequence $w_h$ exists.
 Then we have $w_h \rightharpoonup v_h^1$ in $H^1(\Omega)$ for $h\to0$ and it
holds
\begin{align*}
\|v_h^1-v^1\|_{H^1(\Omega)^n} 
\leq \|v_h^1-w_h\|_{H^1(\Omega)^n} + \|w_h-v^1\|_{H^1(\Omega)^n}.
\end{align*} 
Now we proceed with
\begin{align*}
  C &\|v_h^1 -  w_h\|_{H^1(\Omega)^n}^2\\
  &\leq B(v_h^1-w_h,v_h^1-w_h)
  = B(v_h^1,v_h^1-w_h) - B(w_h,v_h^1-w_h)\\
  &\leq  (B_Vu_h^1,v_h^1-w_h) 
   + (\mu_h^1\nabla \varphi_h^0,v_h^1-w_h) 
   + (\rho_h^0 g,v_h^1-w_h)\\
   &-a(\rho^1_hv^0 + J_h^1,v_h^1,v_h^1-w_h)
    - \tau^{-1}\left( \frac{\rho_h^1 + \rho_h^0}{2}v_h^1 - \rho_h^0v^0,v_h^1-w_h\right)\\
    &+(v_h^1,v_h^1-w_h) - 2(\eta_h^1Dw_h:D(v_h^1-w_h)) - (w_h,v_h^1-w_h)\\
   &\leq \|u_{v,h}^1\|_{L^2(\Omega)} \|v_h^1-w_h\|_{L^2(\Omega)}
   + \|\mu_h^1\nabla \varphi_h^0\|_{L^{\frac{3}{2}}(\Omega)}\|v_h^1-w_h\|_{L^3(\Omega)}\\
    &+ \|\rho_h^0g\|_{L^2(\Omega)^n}\|v_h^1-w_h\|_{L^2(\Omega)}  \\
    & +|a(\rho^1_hv^0 + J_h^1,v_h^1,v_h^1-w_h)|\\
    & + \tau^{-1}\|\frac{1}{2}(\rho_h^1 + \rho_h^0)v_h^1 -
    \rho_h^0v^0\|_{L^2(\Omega)}\|v_h^1-w_h\|_{L^2(\Omega)}\\
    & + \|v_h^1-w_h\|_{L^2(\Omega)}^2 + |2(\eta_h^1Dw_h:D(v_h^1-w_h))|.  
\end{align*}
Now $|(\eta_h^1Dv^1:D(v_h^1-w_h))| \to 0$ for $h\to 0$ since $\eta_h^1Dw_h \to \eta^1Dv^1$ in
$L^2(\Omega)$ and $D(v_h^1-w_h) \rightharpoonup 0$ in $L^2(\Omega)$, and thus beside the
trilinear form all terms directly vanish for $h\to 0$.

For the trilinear form we use the antisymmetry $a(\cdot,v_h^1-w_h,v_h^1-w_h) = 0$ and proceed
\begin{align*}
   \left|a(\rho^1_hv^0 + J_h^1,v_h^1,v_h^1-w_h)\right|
  & = \left|a(\rho^1_hv^0 + J_h^1,w_h,v_h^1-w_h)\right|\\
  &= \left|\frac{1}{2}(t_h\nabla w_h,v_h^1-w_h)- \frac{1}{2}(t_h \nabla (v_h^1-w_h),w_h)\right|\\
  & \leq \frac{1}{2} \|t_h\|_{L^3(\Omega)}\|w_h\|_{H^1(\Omega)}\|v_h^1-w_h\|_{L^2(\Omega)}\\
 & +\frac{1}{2}
 \left|(t_h \nabla (v_h^1-w_h),w_h)\right|,
\end{align*}
We note the strong convergence $w_h \to v^1$ in $L^6(\Omega)$ and $t_h\to t$ in $L^3(\Omega)$.
Thus the last term tends to zero for $h\to 0$.

For $m>1$ we use $\rho_h^{m-1} \to \rho^{m-1}$ in $L^\infty(\Omega)$ to again obtain the strong
convergence $\rho^{m-1}_hv_h^{m-1} + J_h^{m-1} \to \rho^{m-1}v^{m-1} + J^{m-1}$ in $L^3(\Omega)$.

\medskip

\noindent\textbf{The limit $h\to0$ in the dual equation}\\
The convergence of \eqref{eq:FD:adj:Dv} and \eqref{eq:FD:adj:Dv_init} to
\eqref{eq:TD:adj:Dv} and \eqref{eq:TD:adj:Dv_init}, i.e. the adjoint
Navier--Stokes equation, is shown as in the primal equation using the strong
convergence of $\varphi_h^m$ in $L^\infty(\Omega)$ and $\mu_h^m$ in
$W^{1,3}(\Omega)$ to show convergence of the trilinear form and of the 
diffusion term.

The convergence of \eqref{eq:FD:adj:Dmu} and \eqref{eq:FD:adj:Dmu_init} to
\eqref{eq:TD:adj:Dmu} and \eqref{eq:TD:adj:Dmu_init} uses strong convergence of
$p_{v,h}^{m+1}$ in $L^4(\Omega)$ and of $v_h^{m+1}$ in $L^6(\Omega)$, where
the additional regularity for $v_h$ is required.

The convergence of \eqref{eq:FD:adj:Dphi} and \eqref{eq:FD:adj:Dphi_init}
to \eqref{eq:TD:adj:Dphi} and \eqref{eq:TD:adj:Dphi_init} also follows directly
using the above shown strong convergence of the primal variables.
Especially for the term 
$(\eta^\prime \tilde \varphi Dv_h^{m+1}:Dp_{v,h}^{m+1})$
we need the strong convergence $v_h^{m+1} \to v^{m+1}$ in $H^{1}(\Omega)$.

\medskip

\noindent\textbf{The limit $h\to0$ in the derivative w.r.t. the control}\\
The convergence of \eqref{eq:FD:opt_UV} to \eqref{eq:TD:opt_UV} is shown using the
strong convergence $p_{v,h}^m$ in $L^2(\Omega)$.

The convergence of \eqref{eq:FD:opt_UB} to \eqref{eq:TD:opt_UB} is shown
using the various strong convergence results.

Finally we show the convergence of \eqref{eq:FD:opt_UI} to \eqref{eq:TD:opt_UI}.
Since $J(u_h,\varphi_h) \to J(u^\star,\varphi^\star)$, 
we observe convergence $\|\nabla u_{I,h}\|_{L^2(\Omega)} \to \|\nabla u_I^\star\|_{L^2(\Omega)}$.
Together with Poincar\'e's inequality and the weak convergence 
$u_{I,h} \rightharpoonup u_I^\star$ in $H^1(\Omega)$ we observe strong convergence $u_{I,h} \to
u_I^\star$ in $H^1(\Omega)$.
The convergence \eqref{eq:FD:opt_UI} to \eqref{eq:TD:opt_UI} now readily follows.

%------------------------------------------------------------------------------
%-------------------------------------------------------------------------------
%-----------------------The numerical section-----------------------------------
%-------------------------------------------------------------------------------
%-------------------------------------------------------------------------------

\section{Numerical examples}\label{sec:num}

In this section we show numerical results for the optimal control problem
\ref{prob:FD:Ph}. The implementation is done in \verb!C++! using the
finite element toolbox FEniCS \cite{fenics_book} together with the PETSc linear algebra
backend  \cite{petsc_webpage} and the linear solver MUMPS \cite{mumps}. For the
adaptation of the spatial meshes the toolbox ALBERTA \cite{alberta_book} is used.
The minimization problem is solved by steepest descent method. If the
initial phase field is not used as control, we use the GNU scientific library \cite{gsl_webpage},
if the initial value is used as
control we use a self written implementation using the $H^1$ regularity of
the control $u_1$.

Let us next  define some data, that is used throughout all examples.
We use
$\rho(\varphi) = \frac{\rho_2-\rho_1}{2}\varphi +
\frac{\rho_1+\rho_2}{2}$ and
$\eta(\varphi) = \frac{\eta_2-\eta_1}{2}\varphi +
\frac{\eta_1+\eta_2}{2}$, where $\rho_1, \rho_2$ and $\eta_1,\eta_2$ depend on
the actual example.
For the free energy we always use \eqref{eq:freeEnergy},
with $s =1e4$, and the mobility is set to $b\equiv \epsilon/500$.

\subsection{The adaptive concept}
For the construction of the spatially adapted meshes we use the error indicators
that are constructed in \cite{GarckeHinzeKahle_CHNS_AGG_linearStableTimeDisc}
for the primal equation and use the series of meshes that we construct for the
primal equation also for the dual equation.
This means that we use classical residual based error estimation to obtain
suitable error indicators.
We note that following \cite{CarstensenVerfuerth_EdgeResidualDominate} the
cell-wise residuals for the Cahn--Hilliard equation can be subsumed to
the edge-wise error indicators. We further note that from our numerical tests we
obtain that the cell-wise residuals of the momentum equation is much smaller
than the edge-wise indicators, while it turns out to be very expensive to
evaluate. Thus we neglect this term.
The final error indicator is the cell-wise sum of the jumps of the normal
derivatives of the phase field variable, the chemical potential and the velocity
field over the cell boundary.
The final adaptation scheme for the primal equation is a
D\"orfler marking scheme based on this indicator, see e.g.
\cite{Doerfler,GarckeHinzeKahle_CHNS_AGG_linearStableTimeDisc}.

% In \cite{GarckeHinzeKahle_CHNS_AGG_linearStableTimeDisc}
% a post processing of the marked cells is introduced that guarantees that an
% energy inequality is fulfilled in every time step. This is required due to the fact that coarsening the mesh
% might locally generate energy. In this work we do not use this post
% processing.
% In \cite{GarckeHinzeKahle_CHNS_AGG_linearStableTimeDisc} we obtain that the energy
% inequality is violated very seldom while this postprocessing increases the
% total number of cells in the mesh. Since here we have to hold the full primal
% solution in memory to evaluate the dual equation we try to minimize the required
% amount of cells.

For the D\"orfler marking we set the largest cell volume to $V_{max} = 0.0003$,
while the smallest cell volume is set to
$V_{min} =\frac{1}{2}\left(\frac{\pi\epsilon}{8}\right)^2$ which results in 8
triangles across the interface of thickness $\mathcal{O}(\pi\epsilon)$.

Concerning the temporal resolution, we stress that we did not discretize the
control $u_V$ and $u_B$ with respect to time, i.e. we use the variational
discretization approach from \cite{HinzeVariDisk}.
Thus we can adapt the time step size during the optimization to fulfill a CFL-condition without changing the
actual control space.
Thus we start with a given large time step size $\tau$ and reduce this steps size whenever the
CFL-condition $\max_{T} \frac{|y^m|_T|\tau}{\mbox{diam}(T)} \leq 1$ is violated
for any $m=1,\ldots,M$ by halven $\tau$.

\subsection{A rising bubble}
%nRB aB
In this example investigate the pure boundary control $\alpha_V \equiv
\alpha_I \equiv 0$.
Here we use $u_I = \varphi_0$ as given data that we represent on a
adapted mesh using the proposed adaptive concept.

We investigate the example of a rising bubble, compare
\cite{kahle_dissertation} and use the parameters from the benchmark paper
\cite{Hysing_Turek_quantitative_benchmark_computations_of_two_dimensional_bubble_dynamics},
i.e.
$\rho_1 = 1000$, $\rho_2 = 100$, $\eta_1 = 10$, $\eta_2 = 1$.
The surface tension is $24.5$ which due to our choice of free energy corresponds
to $\sigma = 15.5972$. The gravitational constant is $g=(0,-0.981)^t$ and the
computational domain is $\Omega = (0,1)\times (0,1.5)$.
The time interval is $I=[0,1.0]$ 
and we start with a step size $\tau = 5e-3$, that is refined to $\tau = 2.5e-3$ throughout the
optimization.

The initial phase field is given by
\begin{align}
  \label{eq:num:c2s:phi0}
  \varphi_0(x) =
  \begin{cases}
\sin((\|x-M_1\|-r)/\epsilon) & \mbox{ if } |\|x-M_1\|-r|/\epsilon \leq \pi/2,\\
\mbox{sign}(\|x-M_1\|-r) & \mbox{ else,}
\end{cases}
 \end{align}
  with $M_1 =(0.5,0.75)^t$ and $r=0.25$. The desired phase field is given by the same
expression but with $M_1 = (0.5,0.5)^t$. Thus we aim to move a bubble to the
bottom without changing its shape.

Concerning the ansatz functions for the operator $B_B$ we introduce
the vector field
\begin{align*}
  (f[m,\xi,c](x))_i =
  \begin{cases}
\cos \left( (\pi/2) \|\xi^{-1}(x-m) \| \right)^2 &
\mbox{ if } c \equiv i \mbox{ and } \|\xi^{-1}(x-m) \|\leq 1,\\
0 & \mbox{else.}
\end{cases}
\end{align*}  
This describes an approximation to the Gaussian bell with local support. The
center is given by $m$ and the diagonal matrix $\xi$ describes the width of
the bell in unit directions. We identify a scalar value for
$\xi$ with $\xi I$, where $I$ denotes the identity matrix.
The parameter $c$ is the number of  the component in which the vector field $f$
is non-zero.
On the left and right boundary of $\Omega$ we provide 10 equidistantly
distributed ansatz functions $f[m_i,\xi_i,c_i](x)$. 
Here $\xi_i = 1.5/10$ and $\xi_i = 1.0/10$ if $m_i$ is located on bottom or top.
We always choose $c_i$ such that the ansatz function is tangential to $\Omega$.

We set $\alpha=1e-10$ and  $\epsilon = 0.04$
and stop the optimization as soon as $\|\nabla J(u)\|_U$ is decreased by a
factor of 0.1.

In Figure \ref{fig:num:nRB:phi0_phid_Bu} we present the initial phase field
$\varphi_0$, the desired phse field $\varphi_d$ and the control areas together
with the zero-level lines of $\varphi_0$ and $\varphi_d$.

\begin{figure}
  \centering
  \fbox{
  \includegraphics[width=0.25\textwidth]{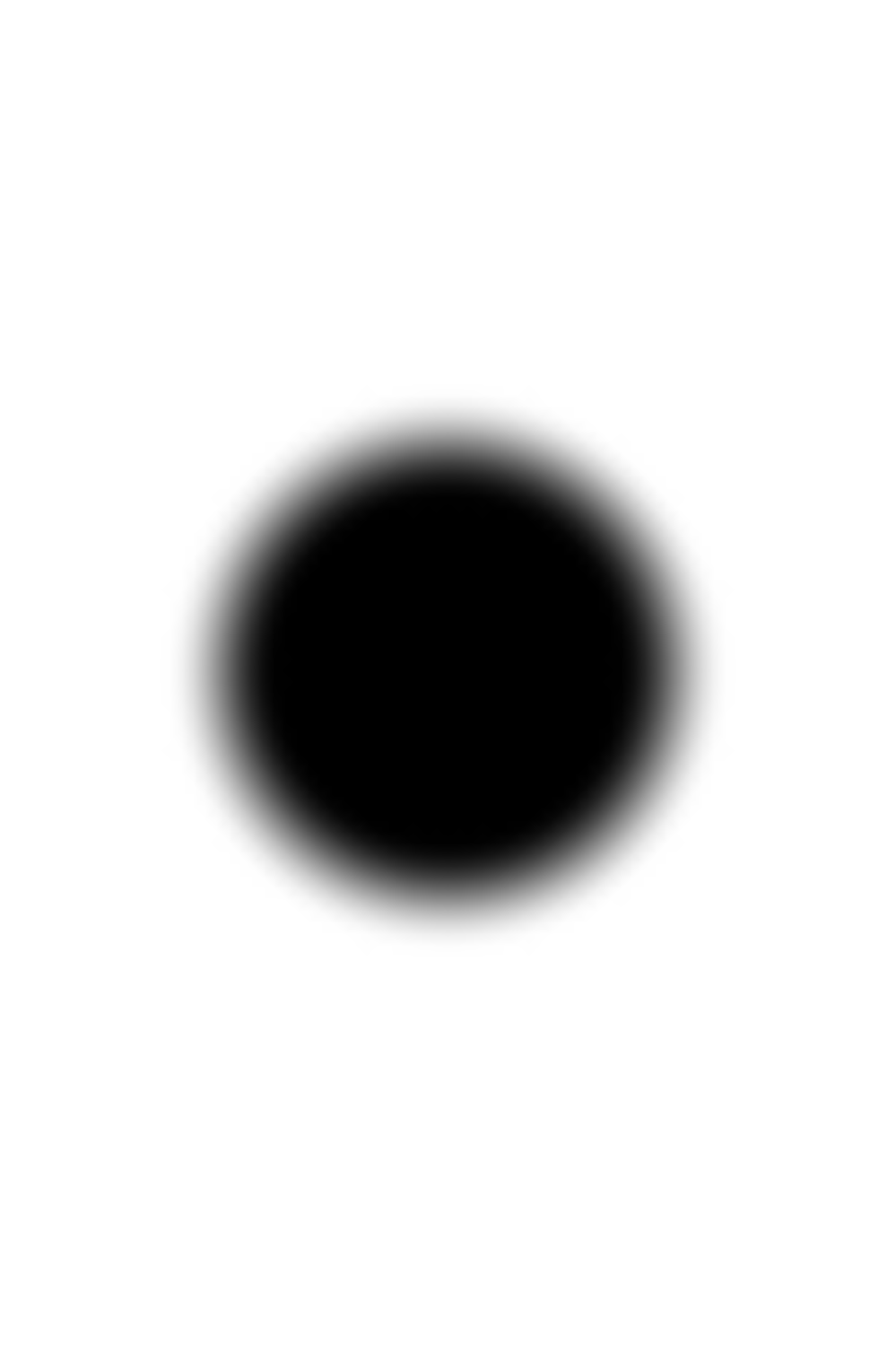}
  }
  \hfill
  \fbox{
  \includegraphics[width=0.25\textwidth]{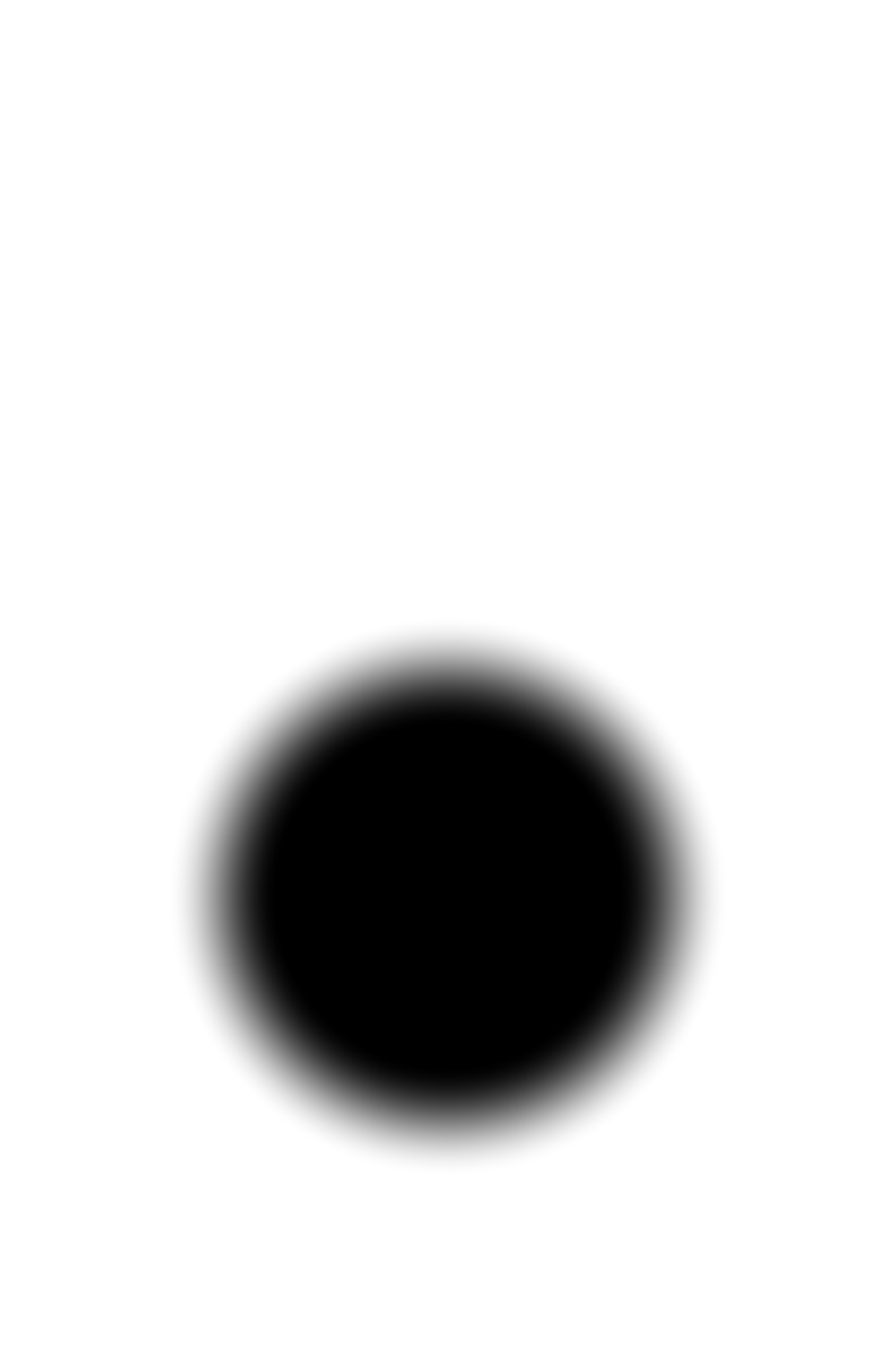}
  }
  \hfill
  \fbox{
  \includegraphics[width=0.25\textwidth]{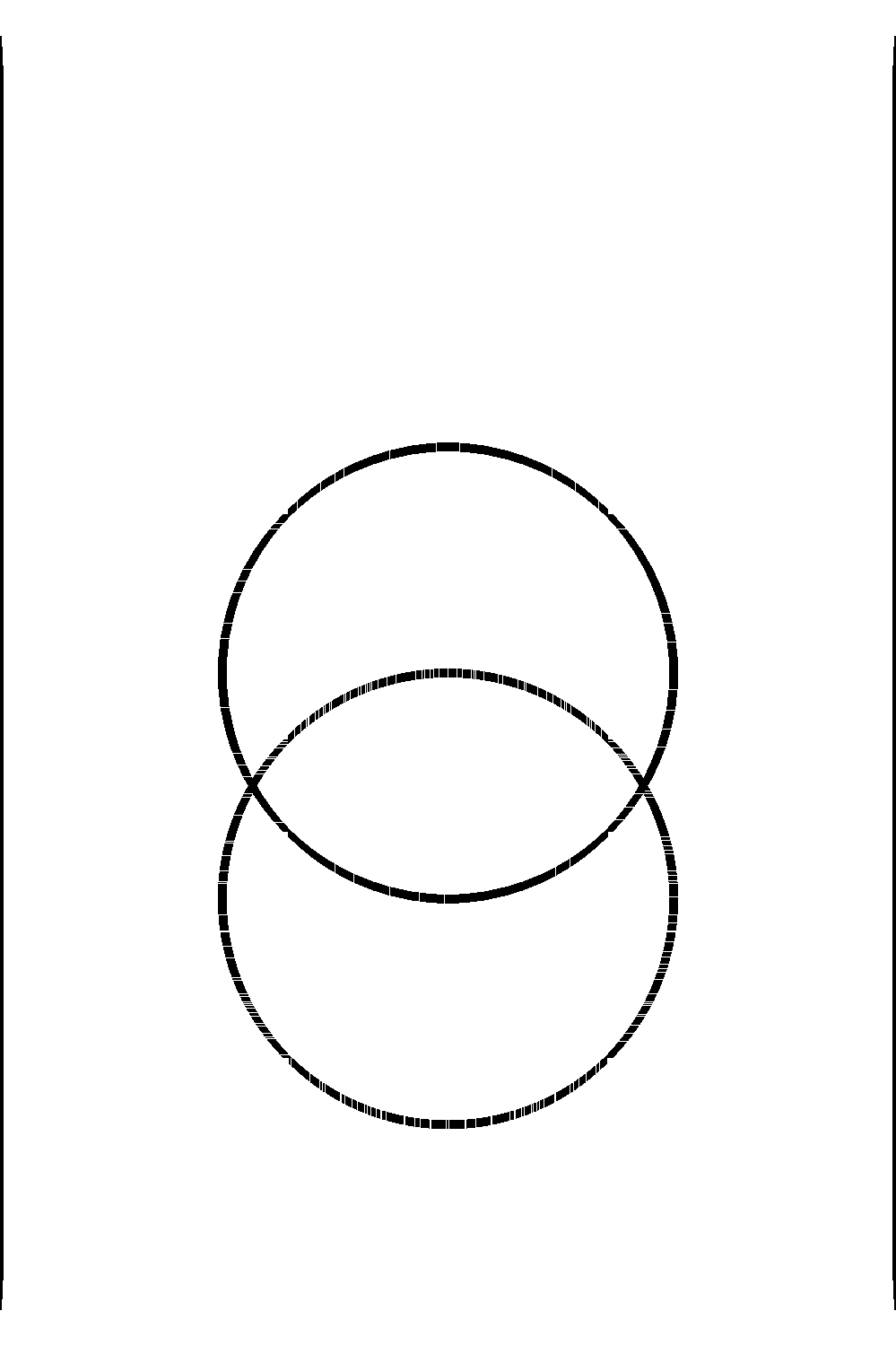}
  }
  
  \caption{The initial phase field $\varphi_0$ (left), the desired phase field
  $\varphi_d$ (middle) and the control areas together with the zero-level lines
  of $\varphi_0$ and $\varphi_d$ (right) for the rising bubble example.
  Note that each of the control areas contains 10 controls of the type
  $f[m,\xi,c](x)$ that point tangential to $\partial \Omega$ with overlapping
  support.
  }
  \label{fig:num:nRB:phi0_phid_Bu}
\end{figure}

The steepest descent method is able to reduce $\|\nabla J\|_{U}$
from $6e-2$ to $4.6e-2$ in 67 iterations and stagnates due to no further
decrease in $\|\nabla J\|_U$. Mean while the functional 
 $J$ is reduced from $0.509$ to $0.033$.
In Figure \ref{fig:num:nRB:phiOpt_lr} we show the evolution of $\varphi$ for the
optimal control together with the magnitude of the velocity field.

In Figure \ref{fig:num:nRB:nu_over_time} we show the evolution of the
control action over time. We observe a rapid decay of the control strength
at the end of the time horizon, while the first peak corresponds to a strong
control at the side walls in the region above the bubble, that is rather
inactive after this initial stage.

\begin{figure}
  \centering
  \fbox{
  \includegraphics[width=0.2\textwidth]{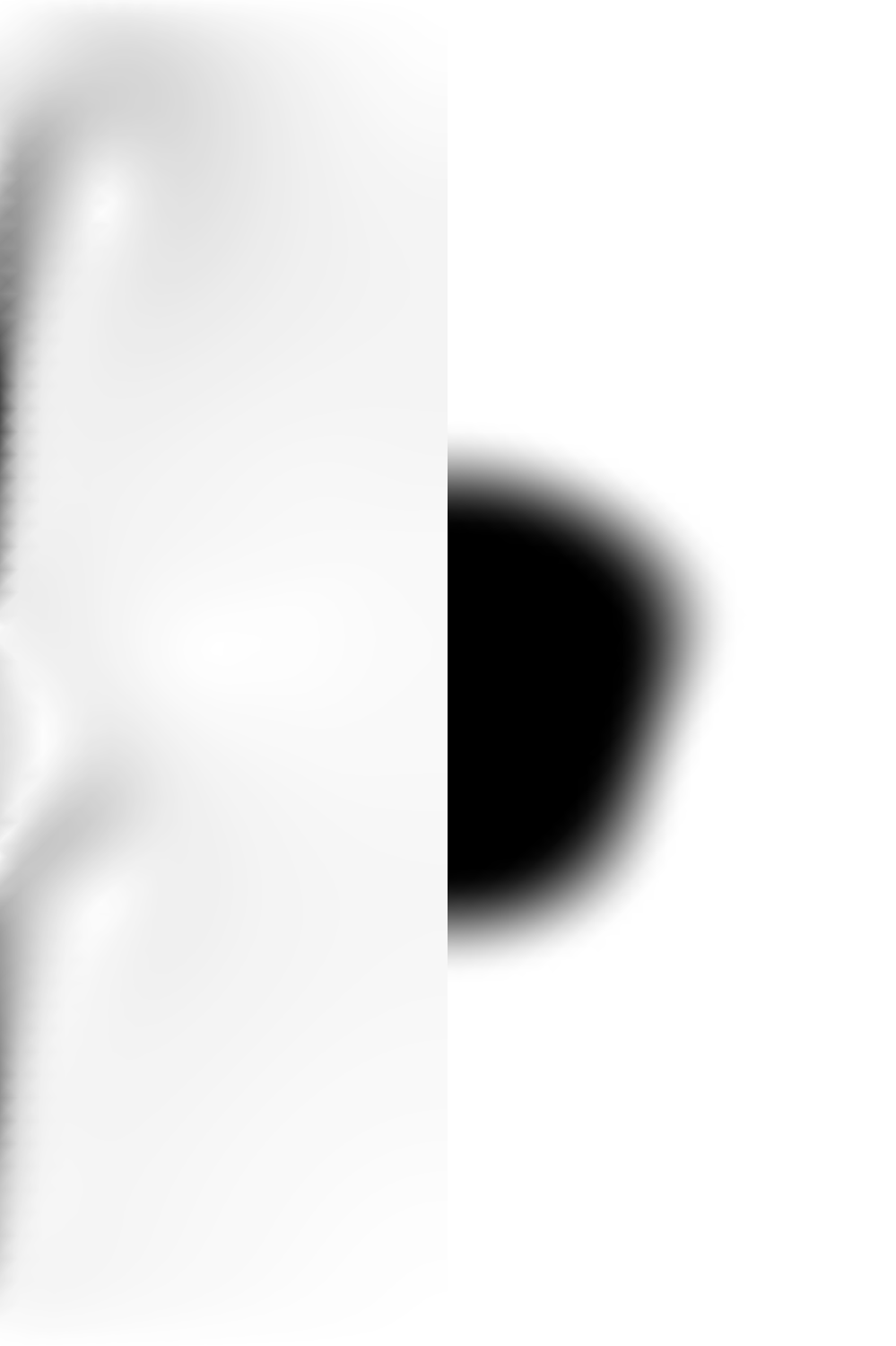}
  }
  \hfill
  \fbox{
  \includegraphics[width=0.2\textwidth]{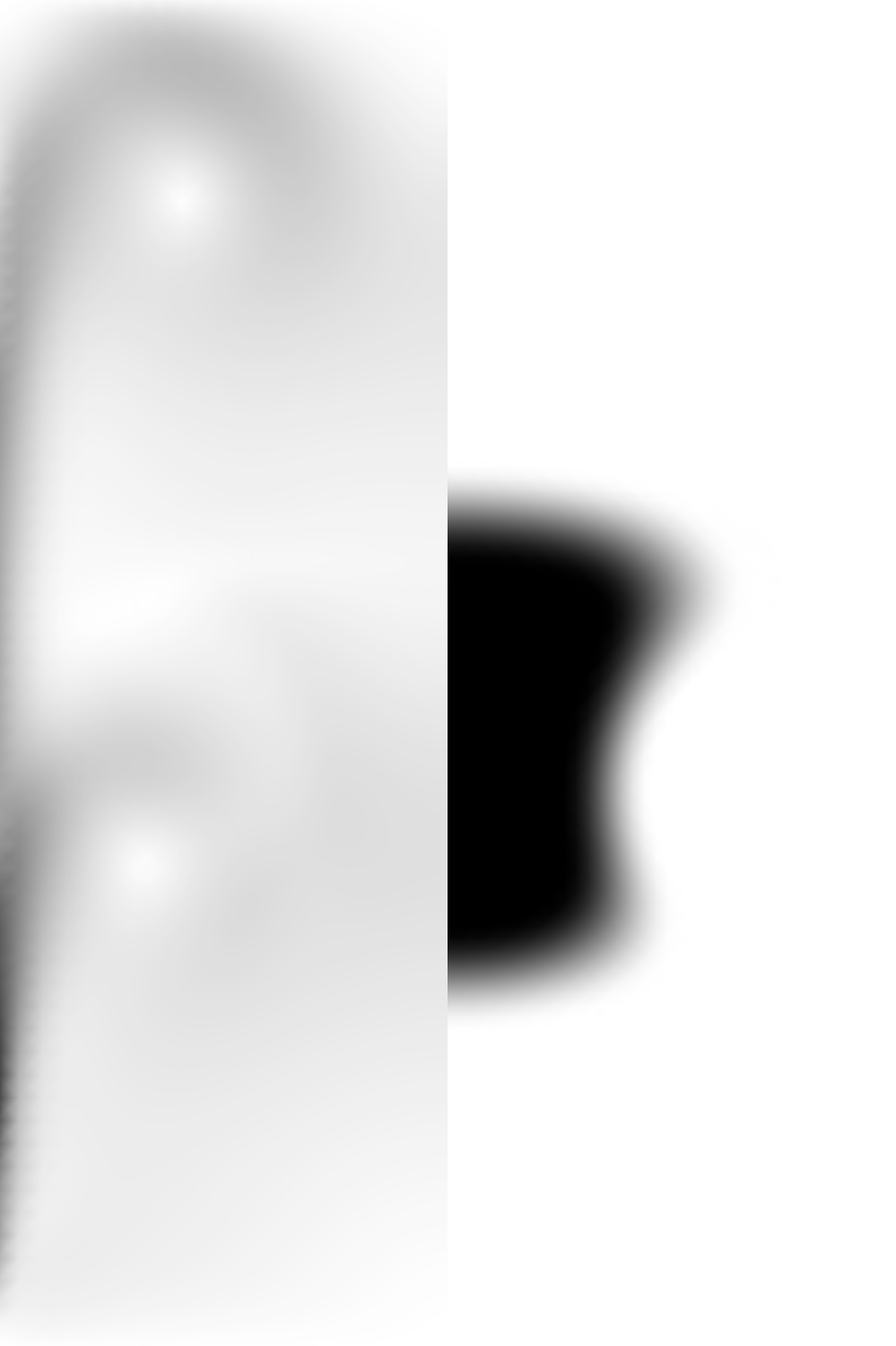}
  }
  \hfill
  \fbox{
  \includegraphics[width=0.2\textwidth]{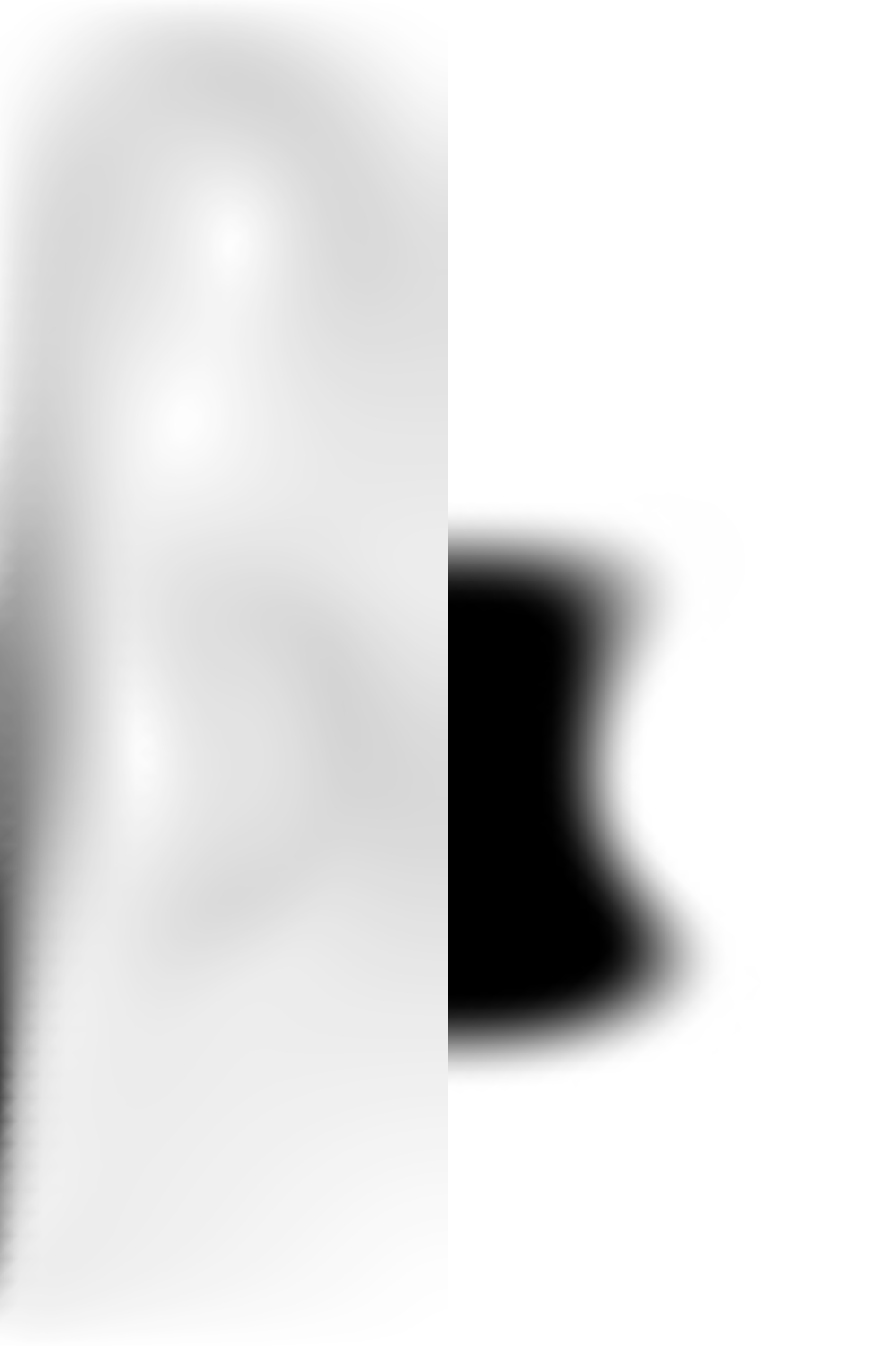}
  }
  \hfill
  \fbox{
  \includegraphics[width=0.2\textwidth]{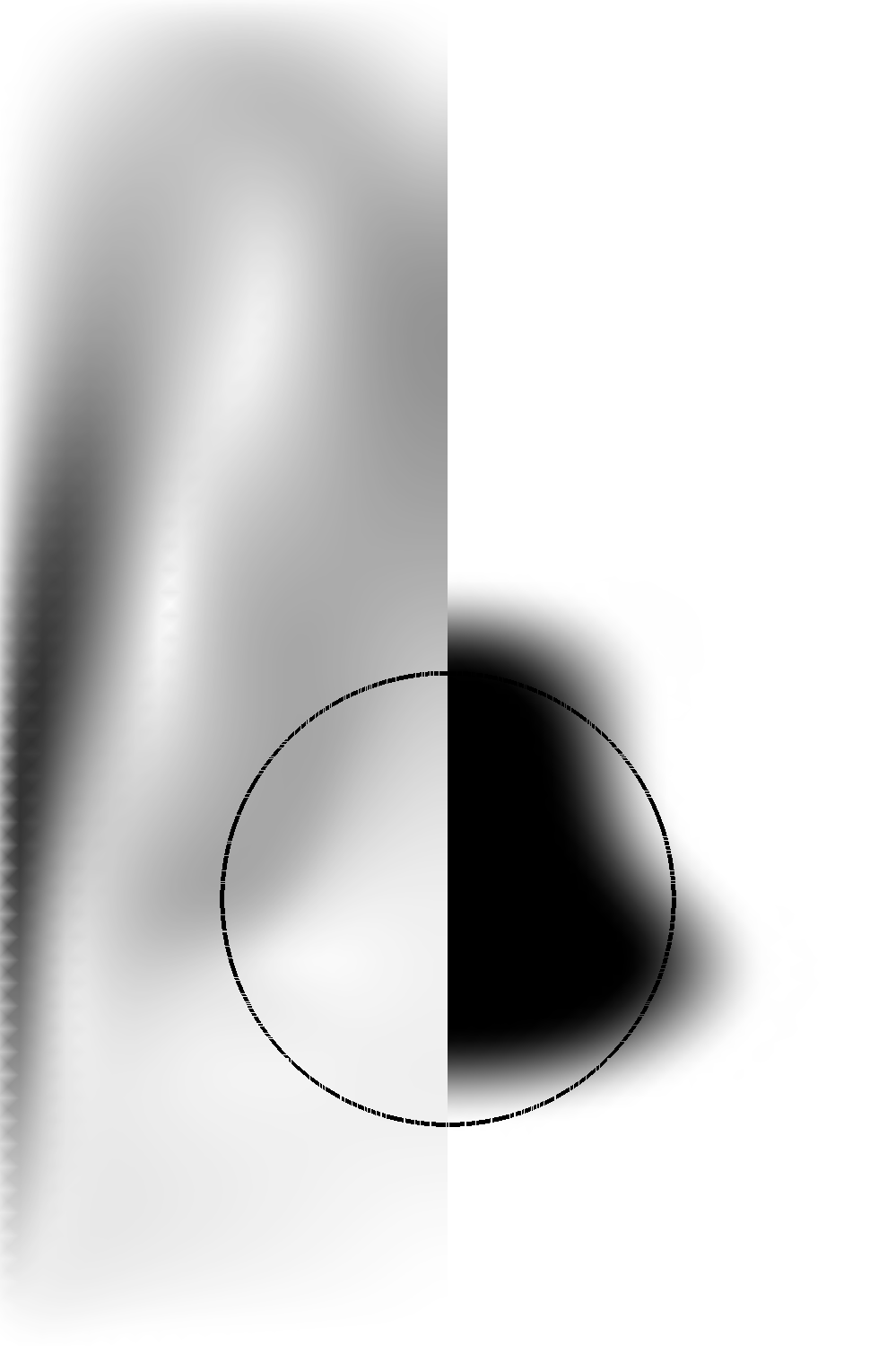}
  }
 
  \caption{
  The evolution of the optimally controlled phase field and velocity field at
  times $t=0.25,0.5,0.75,1.0$ (left to right) when control is only applied
  to the side walls and not at the bottom and the top part of the boundary.
  The pictures show the magnitude of the velocity field on the left and the phase field on the right.
  For $t=0.1$ we additionally indicate the zero-level line of $\varphi_d$ by a
  black line. Note that the velocity field coincides with $B_Bu_B$ on the
  boundary. }
  \label{fig:num:nRB:phiOpt_lr}
\end{figure}

\begin{figure}
  
  \centering
  \input{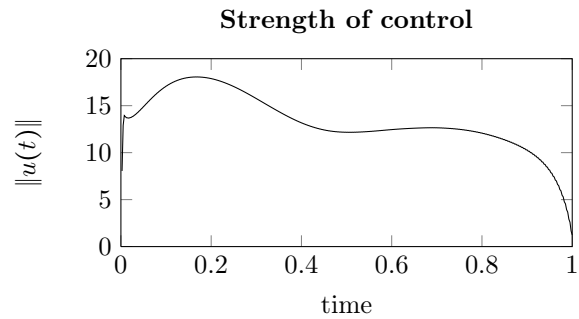}

  \caption{The evolution of the optimal control action over time, i.e.
  $\|u(t)\|$ for the rising bubble example.}
  \label{fig:num:nRB:nu_over_time}
\end{figure}

\subsection{Reconstruction of the initial value}
% section identi: fs
%fs aI
Finally we investigate an example of finding an initial phase field, such that
after a given amount of time without further control action a desired phase
field is achieved.
Here we apply only initial value control, i.e. $\alpha_V = \alpha_B = 0$, and we
use no-slip boundary conditions for the velocity field.

Let us turn to the representation of $u_I$.
We initialize $u_I$ with a constant value $u_I = -0.8$ and use a homogeneously
refined initial mesh for its representation. We use this mesh for $\mathcal T_1$.

 After each step of
the minimization algorithm we use the jumps accross edges in normal direction of $\nabla u_I$  to construct  a new
grid for the representation of $u_I$ and interpolate the current control
to the new grid. The marking is evaluated based on a D\"orfler approach.

The parameter for this example are given as $\rho_1 = 1000$,
$\rho_2 = 1$, $\eta_1 = 10$,
$\eta_2 = 0.1$, $\sigma = 1.245$ and $g\equiv -0.981$.
These are the parameters of the second benchmark from
\cite{Hysing_Turek_quantitative_benchmark_computations_of_two_dimensional_bubble_dynamics}, where
$\sigma$ was rescaled due to our specific choice of energy.
We note that due to the large
ratio in density, the bubble undergoes strong deformation during rising.
The optimization
horizon again is $I = [0,1.5]$, and $\Omega = (0,1)^2$.
We set $\alpha = 0.2$ and 
 solve the  optimization problem for $\epsilon = 0.02$.

We initialize the optimization with $u_I \equiv -0.8$ 
and use a circle around $M = (0.5,0.6)$ with radius $r = 0.1763040551$ as defined in \eqref{eq:num:c2s:phi0}
% \begin{align}
%   \label{eq:num:fs:phiD}
%   \varphi_d(x) =
%   \begin{cases}
% \sin((\|x-M\|-r)/\epsilon) & \mbox{ if } |\|x-M\|-r|/\epsilon \leq
% \pi/2,\\
% \mbox{sign}(\|x-M_1\|-r) & \mbox{ else}
% \end{cases}
% \end{align}
as desired shape. These values are used such that $\int_\Omega \varphi_d - u_I
\dx = 0$ is fulfilled.

The optimization problem is solved using the VMPT method, proposed in
\cite{Blank_Rupprecht__projected_gradient_in_Banach_space}. It is an extension of the projected
gradient method to the Banach space setting. In our situation this is
$H^1(\Omega)\cap L^\infty(\Omega)$.

We stop the allover algorithm as soon as $|(DJ_{u_I}(\cdot), v)| < 1e-3 $, where $v$ denotes the
current normalized search direction.
In our example this is reached after 31 iterations, where $J$ is reduced from 3.8e-1 to 1.9e-1, and
especially $\|\varphi^K-\varphi_d\|$ is reduced from 0.43 to 0.16.

In Figure \ref{fig:num:fs:opt} we show the initial shape at the end of the
optimization process, on the left and the corresponding shape at the end of
the optimization time interval together with the zero level line of the desired
shape on the right.

\begin{figure} 
  \centering
  \fbox{
  \includegraphics[width=0.35\textwidth]{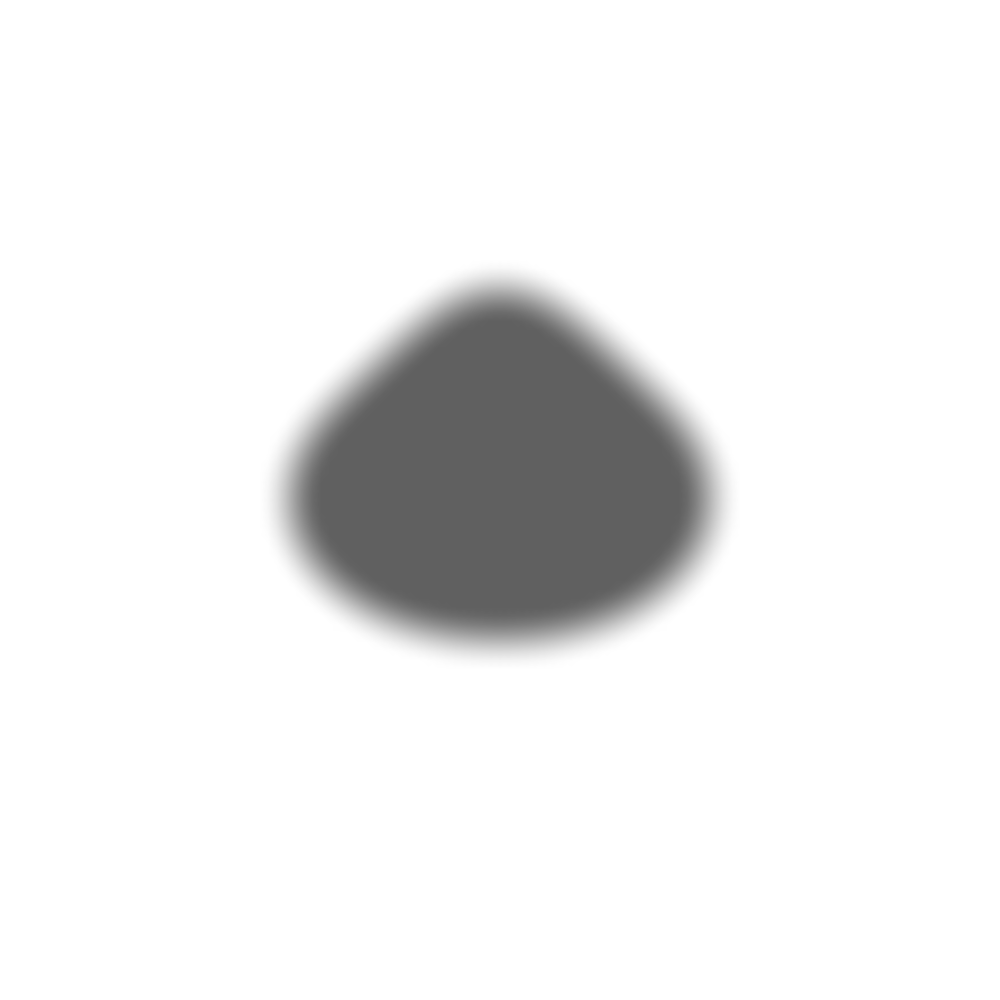}
  }
  \hspace{1em}
  \fbox{
  \includegraphics[width=0.35\textwidth]{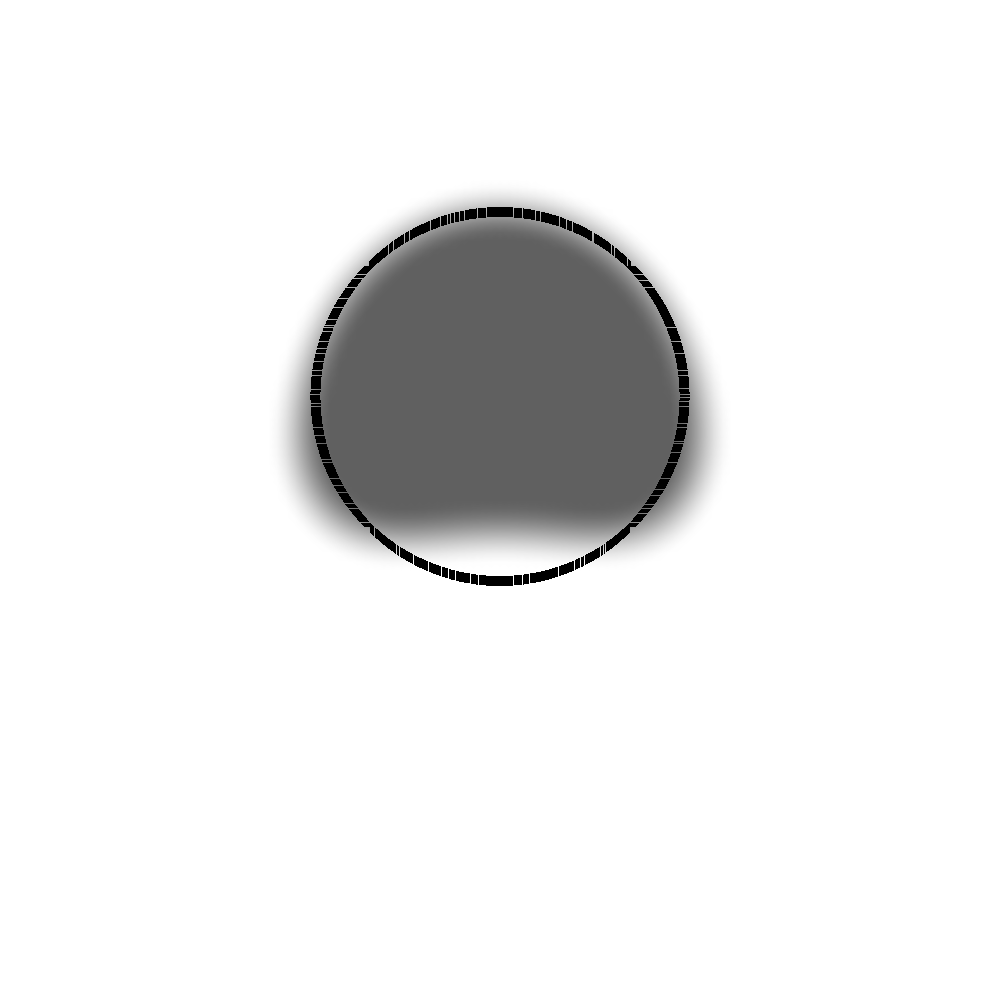}
  }
  \caption{The optimal control $u_I$ (left) and the resulting distribution at
  the end of the time interval (right), where the bubble is shown in gray. The black line indicates
  the zero level line of the desired shape.}
  \label{fig:num:fs:opt}
\end{figure}

\begin{remark}
  In first examples we used an energy for $W_u$ that fulfills
  Assumptions~\ref{ass:psi_C2}--\ref{ass:phi_meanZero} and the method of steepest descent to solve
  the resulting optimization problem. There we only got very slow convergence of the algorithm and
  the resulting optimal $u_I$ had much broader interfaces. So it seems that it is recommended to use
  the non-smoth free energy as we propose here.
\end{remark}

\bibliographystyle{apalike}

\end{document}